\documentclass[a4paper]{article}
\usepackage{amsmath,amssymb,a4wide}
\usepackage{color,graphicx,epstopdf}
\usepackage[a4paper]{geometry}
\usepackage{nicefrac}
\usepackage{lineno} 

\newcommand{\bu}{\boldsymbol u}
\newcommand{\bv}{\boldsymbol v}
\newcommand{\bw}{\boldsymbol w}
\newcommand{\bz}{\boldsymbol z}

\newcommand{\bx}{\boldsymbol x}
\newcommand{\be}{\boldsymbol e}

\newcommand{\bV}{\boldsymbol V}
\newcommand{\bE}{\boldsymbol E}
\newcommand{\bs}{\boldsymbol s}

\newcommand{\bg}{\boldsymbol g}
\newcommand{\btau}{\boldsymbol \tau}
\newcommand{\ee}{\mathrm e}
\newcommand{\bff}{\boldsymbol f}

\newcommand{\di}{\mathrm d}
\renewcommand{\(}{\left(}
\renewcommand{\)}{\right)}
\newcommand{\esssup}[1]{\mathrm{ess}\sup_{\hspace*{-1.5em} #1}}
\newcommand{\todo}[1]{{\color{red} #1}}

\newtheorem{Theorem}{Theorem}
\newtheorem{lema}{Lemma}
\newcounter{remark}
\def\theremark {\arabic{remark}}
\newenvironment{remark}{\refstepcounter{remark}\par\noindent{\bf Remark\ \theremark}\ }{\par}
\newtheorem{Proof}{Proof}

\newenvironment{proof}{\begin{Proof}\rm}{\hfill $\Box$ \end{Proof}}

\usepackage{hyperref}

\title{Error analysis of BDF schemes for
the evolutionary incompressible Navier--Stokes equations}
\author{Bosco
Garc\'{\i}a-Archilla\thanks{Departamento de Matem\'atica Aplicada
II, Universidad de Sevilla, Sevilla, Spain. Research is supported by
Spanish MCINYU under grants PID2021-123200NB-I00 and PID2022-136550NB-I00. (bosco@esi.us.es)}
\and  Volker John\thanks{Weierstrass Institute for Applied Analysis and Stochastics,
Leibniz Institute in Forschungsverbund Berlin e. V. (WIAS), Mohrenstr. 39, 10117 Berlin, Germany.
Freie Universit\"at of Berlin,
Department of Mathematics and Computer Science,
Arnimallee 6, 14195 Berlin, Germany, ({john@wias-berlin.de}).}
  \and Julia Novo\thanks{Departamento de
Matem\'aticas, Universidad Aut\'onoma de Madrid, Spain. Research is supported
by Spanish MINECO
under grant PID2022-136550NB-I00. (julia.novo@uam.es)}}
\date{\today}

\begin{document}

\maketitle
\abstract{Error bounds for fully discrete schemes for the evolutionary incompressible Navier--Stokes equations are derived in this paper.
For the time integration we apply BDF-$q$ methods, $q\le 5$, for which error bounds for $q\ge 3$ cannot be found in the literature.
Inf-sup stable mixed finite elements are used as spatial approximation. 
First, we analyze the standard Galerkin method and second a grad-div stabilized method. 
The grad-div stabilization allows to prove error bounds with constants independent of inverse powers of the viscosity coefficient. We prove optimal bounds for the velocity and pressure with order $(\Delta t)^q$ in time for the BDF-$q$ scheme and order $h^{k+1}$ for the $L^2(\Omega)$ error of the velocity in the first case and $h^k$ in the second  
case, $k$ being the degree of the polynomials in finite element velocity space.}
\medskip


\section{Introduction}

Despite the extensive literature concerning the finite element analysis of the evolutionary incompressible Navier--Stokes equations there are still topics that deserve research.
On the one hand, there are not many papers in which high order methods in time are analyzed. On the other hand, there are also only few papers in which
error bounds with constants independent of inverse powers of the viscosity are proved, so-called robust bounds. The present paper contributes to both aspects.

In the first part of the paper, we prove a priori bounds for fully discrete methods using
the  BDF-$q$ schemes, $q\le 5$, for the time integration and inf-sup stable mixed finite elements for the spatial discretization. The constants in the error bounds of this
part are not independent of $\nu^{-1}$. However, the ideas and methodology are useful for the second part of the paper in which robust bounds  are obtained adding grad-div stabilization to the spatial discretization. 

Error bounds for first and second order methods in time are available in the literature. In \cite{nos_grad_div}, the error analysis of a fully implicit Euler method combined with mixed finite elements with grad-div stabilization is included, and in \cite{nos_bdf}, the authors prove robust bounds for the same spatial discretization and the implicit-explicit BDF-2 method with variable step. 
It is known that BDF methods of order higher than 2 are not $A$-stable. In \cite{YC25} a semi-implicit BDF-3 method for the two-dimensional Navier--Stokes equations with $H^1$ initial data is analyzed. In this paper, only the time discretization is considered while the method is continuous in space; besides, the
bounds are not robust. With this error analysis one cannot account for restrictions on the time step required in the fully discretization of the method, which become then important in practical simulations. 

The motivation of the present paper comes from the recent reference \cite{bdf_stokes}. In this paper, the authors analyze BDF-1 -- BDF-6 for the transient Stokes problem and both inf-sup stable and non-inf-sup stable finite elements. The main ingredient of the analysis of the temporal discretization in \cite{bdf_stokes} is the $G$-stability of BDF schemes,  see  \cite{dalquis1,dalquis2}, which we apply in the present paper to the Navier--Stokes equations. Applying this technique, one can follow the spirit of the error analysis of lower
order time discretization methods to prove a priori bounds for the fully discrete scheme. The price to pay, compared with first or second order methods, is of course a more complicated error analysis with longer proofs  and some technical difficulties. 

As stated above, we start with the error analysis of the non-robust case, which is the standard Galerkin method, since several of the technical difficulties appear already in this case. As a consequence of bounding the nonlinear term of the Navier--Stokes equations,  we find a (mild) time step restriction for $q\ge 3$ that did
not appear in \cite{bdf_stokes}. Concretely, with $\Delta t$ being the time step and $h$ the spatial mesh width, 
we need to assume $\Delta t\le C h^{d/(2q)}$ for the error bound of the velocity and $\Delta t \le C h^{3/(2q)}$ for the error bound of the pressure, where $d \in \{2,3\}$ is the spatial dimension and $C$ is a generic constant that does not depend on the mesh width.
We prove optimal bounds of order $(\Delta t)^q$ in time and $h^{k+1}$ in space for the $L^2(\Omega)$ norm of the error in the velocity and $h^k$ for the $L^2(\Omega)$ norm of the pressure error.

As discussed in \cite{review_nos}, in the situation of small viscosity coefficients and coarse grids, 
which is of high interest in practice, only robust error bounds provide useful information about 
the behavior of a numerical method for smooth solutions; and being able to derive robust error bounds can be considered 
as a necessary condition for a good performance of a method in more realistic problems. 
To prove a robust bound,  we had to deal with additional technical difficulties. To this end, we 
derive a new $G$-stability result for first order finite differences for the BDF-q method, see 
Lemma~\ref{lema_bosco}. With this new result, we can  bound the error of a sum of first order finite differences, which is an extra term appearing in the error analysis of the robust case. This is the 
essential extension to the error analysis of the non-robust case. 
The techniques from \cite{nos_grad_div} can be utilized to bound 
the nonlinear term using the grad-div stabilization. The restrictions of the time step in the robust case are stronger since we need to assume $\Delta t h^{-2}\le C$. This CFL-type condition was already found for the IMEX-BDF-2 method in \cite{nos_bdf}. As known for grad-div stabilized methods, see \cite{review_nos}, the spatial bound for the $L^2(\Omega)$ error of the velocity in the robust case behaves as $h^k$.

Finally, we want to mention that although the BDF-6 method could be handled following the error analysis 
from \cite{bdf_stokes}, we have decided not to include it. As shown in \cite{bdf_stokes}, the analysis of this case is more difficult and must be considered separately and the present paper is already quite
technical and long.

The outline of the paper is as follows. In Section~\ref{sec:notations}, we introduce the problem and 
state some preliminaries and notations.  Section~\ref{sec:bdf} presents results concerning the BDF-q methods. Some prerequisites of the error analysis, including the new $G$-stability lemma, are collected 
in Section~\ref{sec:4}.  Section~\ref{sec:5} contains the error analysis of the non-robust case and Section~\ref{sec:6} the error analysis of the robust case. Finally, some numerical studies are presented in 
Section~\ref{sec:7}. 

\section{Notations, the Navier--Stokes equations, and the Stokes projection}
\label{sec:notations}

Let $\Omega \subset {\mathbb R}^d$, $d \in \{2,3\}$, be a bounded domain with
polyhedral  and Lipschitz boundary $\partial \Omega$. The incompressible
Navier--Stokes equations model the conservation of linear momentum and the
conservation of mass (continuity equation) by
\begin{align}
\label{NS} \bu_t -\nu \Delta \bu + (\bu\cdot\nabla)\bu + \nabla p &= \bff &&\text{in }\ (0,T)\times\Omega,\nonumber\\
\nabla \cdot \bu &=0&&\text{in }\ (0,T)\times\Omega,\\
\bu(0, \cdot) &= \bu_0(\cdot)&&\text{in }\ \Omega,\nonumber\\
\bu(t,\bx)&=0&&\text{in}\ (0,T)\times \partial \Omega\nonumber,
\end{align}
where $\bu$ is the velocity field, $p$ the kinematic pressure, $\nu>0$ the kinematic viscosity coefficient,
$\bu_0$ a given initial velocity, and $\bff$ represents the external body accelerations acting
on the fluid. 

Throughout the paper we will denote by $W^{s,p}(D)$ the Sobolev space of real valued functions defined on the domain $D\subset\mathbb{R}^d$ with distributional derivatives of order up to $s$ in $L^p(D)$, endowed with the usual norm which is denoted by $\|\cdot\|_{W^{s,p}(D)}$. If $s=0$ we understand that $W^{0,p}(D)=L^p(D)$. As it is standard, $W^{s,p}(D)^d$ will be endowed with the product norm that, if no confusion can arise, will be denoted again by $\|\cdot\|_{W^{s,p}(D)}$. We will distinguish the case  $p=2$ using $H^s(D)$ to denote the space $W^{s,2}(D)$. We will make use of the space $H_0^1(D)$, the closure in $H^1(D)$ of the set of infinitely differentiable functions with compact support
in $D$.  For simplicity,
we use $\|\cdot\|_s$ (resp. $|\cdot |_s$) to denote the norm (resp. seminorm) both in $H^s(\Omega)$ or $H^s(\Omega)^d$. The exact meaning will be clear by the context. The inner product of $L^2(\Omega)$ or $L^2(\Omega)^d$ will be denoted by $(\cdot,\cdot)$ and the corresponding norm by $\|\cdot\|_0$. The norm of the space of essentially bounded functions $L^\infty(\Omega)$ will be denoted by $\|\cdot\|_\infty$. For vector-valued function we will use the same conventions as before.
We represent by $\|\cdot\|_{-1}$  the norm of the dual space of $H^1_0(\Omega)$ which is denoted by $H^{-1}(\Omega)$. As usual, we always identify $L^2(\Omega)$
with its dual so we have $H^1_0(\Omega)\subset L^2(\Omega)\subset H^{-1}(\Omega) $ with compact injection.

Using the function spaces
$
V=H_0^1(\Omega)^d$, and \[ Q=L_0^2(\Omega)=\left\{q\in L^2(\Omega):
(q,1)=0\right\},
\]
the weak formulation of problem (\ref{NS}) is:
Find $(\bu,p):(0,T]\rightarrow  V\times Q$ such that for all $(\bv,q)\in V\times Q$,
\begin{equation*}
(\bu_t,\bv)+\nu (\nabla \bu,\nabla \bv)+((\bu\cdot \nabla) \bu,\bv)-(\nabla \cdot \bv,p)
+(\nabla \cdot \bu,q)=(\boldsymbol f,\bv).
\end{equation*}

The Hilbert space
\[
H^{\rm div}=\left\{ \bu \in L^{2}(\Omega)^d, {\nabla \cdot \bu\in L^2(\Omega)} \ \mid \ \nabla \cdot \bu=0, \,
\bu\cdot \mathbf n|_{\partial \Omega} =0 \right\}\]
will be endowed with the inner product of $L^{2}(\Omega)^{d}$ and the space
\[V^{\rm div}=\left\{ \bu \in V \ \mid \ \nabla \cdot \bu=0 \right\}\]
with the inner product of $V$.

The following Sobolev's embedding \cite{Adams} will be used in the analysis: For  $1\le p<\nicefrac{d}s$
let $q$ be such that $\nicefrac{1}{q}
= \nicefrac{1}{p}-\nicefrac{s}{d}$. There exists a positive constant $C$, independent of $s$, such
that
\begin{equation}\label{sob1}
\|v\|_{L^{q'}(\Omega)} \le C \| v\|_{W^{s,p}(\Omega)}, \quad
q' \le q, 
\quad v \in
W^{s,p}(\Omega).
\end{equation}
If $p=\nicefrac{d}s$, then estimate \eqref{sob1} is valid for $1\le q'<\infty$. 
If $p>\nicefrac{d}s$ the above relation is valid for $q'=\infty$. An analog embedding inequality holds for 
vector-valued functions.

In the error analysis, the Poincar\'{e}--Friedrichs inequality
\begin{equation}\label{eq:poin}
\|\bv\|_{0} \leq
C_P\|\nabla \bv\|_{0},\quad\forall\ \bv\in H_0^1(\Omega)^d,
\end{equation}
will be used.

Let $\{V_h\subset V\}_{h>0}$ and $\{Q_h\subset Q\}_{h>0}$ be two families of finite element
spaces composed of piecewise polynomials of
degrees at most $k$ and $k-1$, respectively, that correspond to a family of partitions $\{\mathcal T_h\}_{h>0}$ of
$\Omega$ into mesh cells with maximal diameter $h$. It will be assumed that
the family of triangulations is quasi-uniform.
In the following, we consider pairs of finite element spaces that satisfy the discrete
inf-sup condition
\begin{equation}\label{LBB}
\inf_{q_h\in Q_h}\sup_{\bv_h\in V_h,\bv_h\neq \boldsymbol 0}\frac{(\nabla \cdot \bv_h,q_h)}{\|\nabla \bv_h\|_0\|q_h\|_0}\ge \beta_0,
\end{equation}
with $\beta_0>0$ being a constant independent of the mesh size $h$. The space of discrete divergence-free functions is denoted by
\[
V_h^{\rm div}=\left\{\bv_h\in V_h\ \mid\ (\nabla\cdot \bv_h,q_h)=0\quad \forall\ q_h\in Q_h \right\}.
\]

The following inverse
inequality holds for each $v_{h} \in V_{h}$, see e.g., \cite[Theorem 3.2.6]{Cia78},
\begin{equation}
\label{inv} \| \bv_{h} \|_{W^{m,p}(K)} \leq {c_{\mathrm{inv}}}
h_K^{n-m-d\left(\frac{1}{q}-\frac{1}{p}\right)}
\|\bv_{h}\|_{W^{n,q}(K)},
\end{equation}
where $0\leq n \leq m \leq 1$, $1\leq q \leq p \leq \infty$ and $h_K$
is the size (diameter) of the mesh cell $K \in \mathcal T_h$.

In the sequel,  $I_h \bu \in V_h$ will denote
the Lagrange interpolant of a continuous function $\bu$. The following bound can be found in \cite[Theorem 4.4.4]{brenner-scot}
\begin{equation}\label{cota_inter}
|\bu-I_h \bu|_{W^{m,p}(K)}\le c_\text{\rm int} h^{n-m}|\bu|_{W^{n,p}(K)},\quad 0\le m\le n\le k+1,
\end{equation}
where $n>d/p$ when $1< p\le \infty$ and $n\ge d$ when $p=1$. 
Denoting by $\pi_h$ the $L^2(\Omega)$ projection of the pressure $p$
onto $Q_h$, we have that for $0\le m\le 1$
\begin{equation}\label{eq:pressurel2}
\|p-\pi_h\|_m\le C h^{l+1-m}\|p\|_{l+1}, \quad \forall\ p\in H^{l+1}(\Omega).
\end{equation}

In the analysis, the Stokes problem
\begin{align}\label{eq:stokes_str}
-\nu \Delta \bv+\nabla q&={\bg}\quad\mbox{in }\Omega, \nonumber\\
\bv&=\boldsymbol 0\quad\mbox{on } \partial \Omega,\\
\nabla \cdot \bv&=0\quad\mbox{in } \Omega,\nonumber
\end{align}
will be considered. Let us denote by
$(\bv_h,q_h)\in V_h \times Q_h$  the mixed finite element approximation to \eqref{eq:stokes_str},
given by
\begin{equation}\label{eq:gal_stokes}
\begin{split}
\nu(\nabla \bv_h,\nabla \bw_h)-(\nabla \cdot \bw_h, q_h)&=({\bg},\bw_h)\quad \forall\ \bw_h\in V_h\\
(\nabla \cdot \bv_h,r_h)&=0\qquad\quad\,\,\,  \forall r_h\in Q_h.
\end{split}
\end{equation}
Following \cite{girrav}, one gets the estimates
\begin{align}\label{eq:cota_stokes_v0}
h^{-1}\|\bv-\bv_h\|_0+\|\bv-\bv_h\|_1&\le  C \left(\inf_{\bw_h\in V_h}\|\bv-\bw_h\|_1+\nu^{-1}\inf_{r_h\in Q_h}\|q-r_h\|_0\right),\\
\|q-q_h\|_0&\le C\left( \nu \inf_{\bw_h\in V_h}\|\bv-\bw_h\|_1+
\inf_{r_h\in Q_h}\|q-r_h\|_0\right).\label{eq:cota_stokes_pre}\end{align}
Let  $(\bu,p)$ be the solution of the Navier--Stokes equations \eqref{NS} with $\bu\in V\cap H^{k+1}(\Omega)^d$,
$p \in Q\cap H^{k}(\Omega)$,
$k\ge 0$. In case 
\begin{equation}\label{eq:stokes_rhs_g_1}
{\bg}={\bff}-\bu_t-(\bu\cdot \nabla)\bu,
\end{equation}
the pair $(\bu,p)$ is also the solution of \eqref{eq:stokes_str} with right-hand side \eqref{eq:stokes_rhs_g_1}. We will denote by $({\bf s}_h,p^s_h)$ its
finite element approximation in $V_h\times Q_h$ (the so-called Stokes projection) that, according to \eqref{eq:cota_stokes_v0} and \eqref{eq:cota_stokes_pre}, satisfies
\begin{eqnarray}\label{eq:cotastokes}
\|\bu-{\bs}_h\|_0+h\|\bu-{\bs}_h\|_1&\le& C h^{k+1}\(\|\bu \|_{k+1}+\|p\|_k\),\\
\|p-p_h^s\|_0&\le& C h^{k}\(\|\bu\|_{k+1}+\|p\|_k\), \nonumber
\end{eqnarray}
where the constant $C$ in \eqref{eq:cotastokes} depends on $\nu^{-1}$.
Assuming enough regularity for the domain $\Omega$ one can also prove that
\begin{eqnarray}\label{sh_menos1}
\|\bu-{\bs}_h\|_{-1}\le C h^{\min(k+2,2k)}\(\|\bu \|_{k+1}+\|p\|_k\).
\end{eqnarray}

Following \cite{chenSiam},
one can also obtain the bound
\begin{align}
\|\bu-{\bs}_h\|_\infty&\le C_\infty h\left(\log(1/h)\right)^{\bar r}\left(\|\nabla \bu\|_{W^{1,\infty}}
+\|p\|_{\infty}\right)\label{cotainfty0},
\end{align}
where  $\bar r=0$ for linear elements and $\bar r=1$ otherwise.
Applying \eqref{cotainfty0}, we obtain that there exists $h_0$ such that for $h\le h_0$ 
\begin{eqnarray}\label{cotasinf}
\|\bs_h\|_\infty&\le& \|\bu\|_\infty+\|\bu-{\bs}_h\|_\infty\le \|\bu\|_\infty+C_\infty h\left(\log(1/h)\right)^{\bar r}\left(\|\nabla \bu\|_{W^{1,\infty}}+\|p\|_{\infty}\right)\nonumber\\
&\le& 2\|\bu\|_\infty.
\end{eqnarray}
Utilizing \eqref{inv} and \eqref{eq:cotastokes} yields
\begin{eqnarray}\label{cotasl2d}
\|\nabla \cdot \bs_h\|_{L^{2d}}&\le&  c_{\mathrm{inv}} h^{-(d-1)/2}\|\nabla \cdot \bs_h\|_0=c_{\mathrm{inv}} h^{-(d-1)/2}\|\nabla \cdot (\bs_h-\bu)\|_0\nonumber\\
&\le& C h^{(3-d)/2}(\|\bu\|_{2}+\|p\|_1)\le C(\|\bu\|_{2}+\|p\|_1).
\end{eqnarray}
Using the 
triangle inequality, the inverse
inequality \eqref{inv}, the interpolation estimate \eqref{cota_inter}, \eqref{eq:cotastokes}, and 
the Sobolev embedding \eqref{sob1}, we get
\begin{eqnarray}\label{cotas2d}
\|\nabla \bs_{h}\|_{L^{2d}} &\le&
c_{\mathrm{inv}}h^{-(d+1)/2}\(\|\bs_h-\bu\|_{0}+\|\bu-I_{h}\bu\|_{0}\)+\|\nabla I_h \bu\|_{L^{2d}}
 \nonumber\\
&\le& C h^{(3-d)/2} \(\|\bu\|_2+\|p\|_1\)+C\|\nabla \bu\|_{L^{2d}}\le C \(\|\bu\|_2+\|p\|_1\).
\end{eqnarray}
Using that $\Omega$ is bounded, we also obtain
\begin{equation}\label{cotas2d_dmenos1}
\|\nabla \bs_{h}\|_{L^{2d/(d-1)}} \le C \|\nabla \bs_{h}\|_{L^{2d}} \le  C \(\|\bu\|_2+\|p\|_1\).
\end{equation}

To get error bounds uniform in $\nu$, in \cite{nos_oseen} a modified Stokes projection was introduced.
We observe that considering  again 
the Stokes problem \eqref{eq:stokes_str} but with
 right-hand side
\begin{equation}\label{eq:stokes_rhs_g}
{\bg}={\bff}- \bu_t-(\bu\cdot \nabla)\bu-\nabla p,
\end{equation}
then $(\bu,0)$ 
 is the solution of  \eqref{eq:stokes_str} with
 right-hand side \eqref{eq:stokes_rhs_g}.
Denoting the corresponding
Galerkin approximation in $V_h\times Q_h$ by $({{\bs}^m_h},l_h)$, one obtains from
\eqref{eq:cota_stokes_v0}--\eqref{eq:cota_stokes_pre}
\begin{align}\label{eq:cotanewpro}
\|\bu-{\bs}_h^m\|_0+h\|\bu-{\bs}^m_h\|_1&\le C h^{k+1}\|\bu \|_{k+1},\\
\label{eq:cotanewpropre}
\|l_h\|_0&\le C \nu h^{k}\|\bu\|_{k+1},
\end{align}
where the constant $C$ does not depend on $\nu$.

Following \cite{chenSiam}, the analogue to bound \eqref{cotainfty0} reads
\begin{align}
\|\bu-{\bs}^m_h\|_\infty&\le C_\infty h\left(\log(\left|\Omega\right|^{1/d}/h)\right)^{\bar r}\|\nabla \bu\|_\infty\label{cotainfty0_ro}.
\end{align}
It is also proved in \cite{chenSiam} that
\begin{align}
\|\nabla (\bu-{\bs}^m_h)\|_\infty&\le C_\infty\|\nabla \bu\|_\infty \label{cotainfty1}.
\end{align}
The constant $C_\infty$ in \eqref{cotainfty0_ro} and \eqref{cotainfty1} does not depend on $\nu$. Arguing as in \eqref{cotasinf},
\eqref{cotasl2d} and \eqref{cotas2d_dmenos1} (applying \eqref{eq:cotanewpro} instead of \eqref{eq:cotastokes}), we get 
\begin{eqnarray}\label{cotas_s_h_mod}
\|{\bs}^m_h\|_\infty\le 2\|\bu\|_\infty,\quad
\|\nabla \cdot{\bs}^m_h\|_{L^{2d}}\le C\|\bu\|_2,\quad
\|\nabla {\bs}^m_h\|_{L^{2d/(d-1)}}\le C\|\bu\|_2,
\end{eqnarray}
with $C$ independent of $\nu$. Finally, from \eqref{cotainfty1}, we can infer that 
\begin{equation}\label{nabla_sh_mod}
\|\nabla {\bs}^m_h\|_\infty\le (C_\infty+1)\|\nabla \bu\|_\infty.
\end{equation}

\section{BDF methods}\label{sec:bdf}

For the temporal integration, we apply a $q$-step backward differentiation formulae (BDF). 
This section reviews those properties of BDF methods that are important for our analysis. 
We use the notations from \cite{bdf_stokes}. 

For $T>0$ we consider a uniform partition of $[0,T]$. 
Let us denote the length of the time step 
by $\Delta t=T/M$ and the time instant by $t_n=n\Delta t$, $n=0,\ldots,M$. 
Then, for $u^n=u(t_n)$, the BDF approximation of order
$q$ to the time derivative $u_t$ is given by
\begin{equation}\label{eq_bdf_q}
\overline \partial_q u^n=\frac{1}{\Delta t}\sum_{i=0}^q \delta_i u^{n-i},\ n\ge q,\quad \mbox{with the relation}\quad \delta(\zeta)=\sum_{i=0}^q \delta_i\zeta^i
=\sum_{l=1}^q\frac{1}{l}(1-\zeta)^l
\end{equation}  
being used for determining the constants of the method. 

We recall the results on $G$-stability for BDF schemes of \cite{bdf_stokes} that allow to use energy estimates 
in the error analysis of the methods. The following lemma is stated and proved in \cite[Lemma 3.6]{bdf_stokes}, see also \cite{dalquis1,dalquis2}.
\begin{lema}[$G$-stability for Hilbert spaces]\label{lema_3_6}
Let $\delta(\zeta)=\sum_{i=0}^q \delta_i \zeta^i$ and $\mu(\zeta)=\sum_{j=0}^q\mu_j\zeta^j$ be polynomials that have no common divisor. Assume that the following condition holds
\begin{equation}\label{con_g}
{\rm Re}\frac{\delta(\zeta)}{\mu(\zeta)}>0 \quad\forall\ \zeta\in {\Bbb C},|\zeta|<1.
\end{equation}
Let $H$ be a Hilbert space with associated norm $\|\cdot\|_H$ and let 
$(\cdot,\cdot)_H$
be a semi-inner product on $H$. Then, there exists a symmetric positive definite
matrix $G=[g_{i,j}]\in {\Bbb R}^{q\times q}$ such that for $v_0,\ldots,v_q\in H$ the following bound holds
\begin{equation*}
{\rm Re}\left(\sum_{i=0}^q \delta_i v^{q-i},\sum_{j=0}^q\mu_j v^{q-j}\right)_H\ge \sum_{i,j=1}^qg_{i,j}\(v^i,v^j\)_H-\sum_{i,j=1}^q g_{i,j}\(v^{i-1},v^{j-1}\)_H.
\end{equation*}
\end{lema}

By using the multiplier technique developed in \cite{nevan_ode}, it is possible to apply the concept 
of $G$-stability to the analysis of $A(\alpha)$-stable methods, like BDF for $q\in\{3,4,5\}$. The following 
lemma can be found in \cite[Lemma 3.7]{bdf_stokes}.

\begin{lema}[$G$-stability and multiplier technique]\label{lem:mult_tech}
For $q=1,\ldots,5$ there exits $0\le \eta_q<1$ such that for $\delta(\zeta)=\sum_{l=1}^q (1/l)(1-\zeta)^l$ and $\mu(\zeta)=1-\eta_q\zeta$  condition \eqref{con_g} holds. The values of $\eta_q$ found 
in \cite{nevan_ode,AK15} are given by 
\begin{equation}\label{eta}
\eta_1=\eta_2=0,\ \eta_3=0.0769,\ \eta_4=0.2878,\ \eta_5=0.8097.
\end{equation}
\end{lema}

As in \cite{bdf_stokes}, we define a so-called $G$-norm associated with the semi-inner product $(\cdot,\cdot)_H$ on
$H$. Given
$\bV^n=[v^n,\ldots,v^{n-q+1}]$ with $v^{n-i+1}\in H$, $i=1,\ldots,q$, we define
\begin{equation}\label{norma_G}
\left|\bV^n\right|_G^2=\sum_{i,j=1}^q g_{i,j}\left(v^{n-i+1},v^{n-j+1}\right)_H,
\end{equation}
where $G$ is the matrix from Lemma~\ref{lema_3_6}. In fact, $|\cdot|_G$ is a seminorm on $H^q$. 
Since $G$ is symmetric positive definite, the following inequality holds
\begin{equation}\label{equi_nor}
\lambda_{\rm min}\left|v^n\right|_H^2\le \lambda_{\rm min}\sum_{j=1}^q\left|v^{n-j+1}\right|_H^2\le \left|\bV^n\right|_G^2\le \lambda_{\rm max} \sum_{j=1}^q\left|v^{n-j+1}\right|_H^2,
\end{equation}
where $|\cdot|_H$ is
the semi-norm induced by the semi-inner product $(\cdot,\cdot)$ and 
$\lambda_{\rm min}$ and $\lambda_{\rm max}$ denote the smallest and largest eigenvalues of $G$, respectively.

Finally, we introduce a time-discrete norm by 
\[
\|v\|_{l^2(0,T;H)} = \(\Delta t \sum_{n=0}^M \|v\|_H^2\)^{1/2}.
\]

\section{Error analysis of the method. Preliminary results}\label{sec:4}

The fully discrete method reads as follows.
 Given $\bu_h^0,\ldots,\bu_h^{q-1}$, solve for $q\le n\le M$ and for all $(\bv_h,q_h)\in V_h\times Q_h$ 
 \begin{eqnarray}\label{eq:bdfq_grad_div}
(\overline \partial_q \bu_h^n,\bv_h)+\nu(\nabla \bu_h^{n},\nabla \bv_h)+b(\bu_h^n,\bu_h^n,\bv_h)
-(p_h^{n},\nabla \cdot \bv_h)&&\nonumber\\
+(\nabla\cdot \bu_h^{n},q_h)+\mu(\nabla \cdot \bu_h^{n},\nabla \cdot \bv_h)& =&(\boldsymbol f^{n},\bv_h),
\end{eqnarray} 
where $\mu$ is the grad-div parameter and the notation $\boldsymbol f^{n}$ means $\boldsymbol f(t_{n})$. The discrete convective term is 
given by 
\[
b\(\bu,\bv,\bw\)=\(B(\bu,\bv),\bw\)\quad \forall\ \bu,\bv,\bw\in V,
\]
where,
\[
B(\bu,\bv)=(\bu\cdot \nabla )\bv+\frac{1}{2}(\nabla  \cdot\bu)\bv\quad \forall\ \bu,\bv\in V.
\]
With this definition, we have the well known property
\begin{equation}
\label{eq:skew}
b(\bu,\bv,\bw)= - b(\bu,\bw,\bv)\quad \forall\ \bu,\bv,\bw\in V,
\end{equation}
such that, in particular, $b(\bu,\bw,\bw) = 0$ for all $\bu, \bw\in V$. 
We will apply the following lemma to bound the nonlinear term.
\begin{lema}[Bounds for the convective term]\label{nonli_bound} Let $\bv, \bw, \bz\in V$, then the following bounds hold
\begin{eqnarray}
|b(\bv,\bv,\bw)|&\le& 2\|\bw\|_{\infty}\|\nabla \bv\|_0\|\bv\|_0,\label{cota_non_1}\\
|b(\bv,\bw,\bz)|&\le& \left(C\|\nabla \bw\|_{L^{2d/(d-1)}}+\|\bw\|_\infty\right)\|\nabla \bv\|_0\|\bz\|_0,\label{cota_non_2}\\
|b(\bv,\bw,\bz)|&\le& \left(\|\bv\|_\infty+C\|\nabla \cdot \bv\|_{L^{2d/(d-1)}}\right)\|\nabla \bw\|_0\|\bz\|_0\label{cota_non_3},\\
|b(\bv,\bw,\bz)|&\le& C(\|\nabla \bw\|_{L^{2d}}+\|\bw\|_\infty)\|\bv\|_0\|\nabla \bz\|_0,\label{cota_non_4}\\
|b(\bv,\bw,\bz)|&\le& C\|\nabla \bv\|_0\|\nabla \bw\|_0\|\nabla\bz\|_0.\label{cota_non_5}
\end{eqnarray}	
\end{lema}
\begin{proof}
The proof is performed in a standard way by applying H\"older's inequality, the bound $\|\nabla \cdot \bv\|_0\le \|\nabla \bv\|_0$ for any $\bv\in H_0^1(\Omega)^d$, e.g., see \cite[(3.41)]{volker_book}, the Sobolev embedding \eqref{sob1}, and the
Poincar\'e--Friedrichs inequality \eqref{eq:poin}. To prove \eqref{cota_non_4}, the divergence term is integrated by parts.
\end{proof}

For the error analysis we will apply the following discrete Gronwall inequality taken from \cite[Lemma 3.2]{nos_bdf}.
\begin{lema}[Discrete Gronwall lemma I]\label{gronwall2}
Let $B,a_j,b_j,c_j,\gamma_j,\delta_j$ be nonnegative numbers such that
\begin{equation}\label{eq:gronwall1}
a_n+\sum_{j=0}^n b_j\le \gamma_n a_n +\sum_{j=0}^{n-1} (\gamma_j+\delta_j) a_j+\sum_{j=0}^n c_j+B,\quad for \quad n\ge 0.
\end{equation}
Suppose that $\gamma_j<1$ for all j, and set $\sigma_j=(1-\gamma_j)^{-1}$, then 
\[
a_n+\sum_{j=0}^nb_j\le \exp\left(\sigma_n \gamma_n+\sum_{j=0}^{n-1}(\sigma_j\gamma_j+\delta_j)\right)\left\{\sum_{j=0}^nc_j+B\right\}, \quad for\quad  n\ge 0.
\]
\end{lema}

We will also use the following discrete Gronwall inequality whose proof can be found in \cite{hr4}.
\begin{lema}[Discrete Gronwall lemma II]\label{gronwall}
Let $k,B,a_j,b_j,c_j,\gamma_j$ be nonnegative numbers such that
$$
a_n+k\sum_{j=0}^n b_j\le k\sum_{j=0}^n \gamma_j a_j+k\sum_{j=0}^n c_j+B,\quad for \quad n\ge 0.
$$
Suppose that $k\gamma_j<1$, for all j, and set $\sigma_j=(1-k\gamma_j)^{-1}$. Then
$$
a_n+k\sum_{j=0}^nb_j\le \exp\left(k\sum_{j=0}^n\sigma_j\gamma_j\right)\left\{k\sum_{j=0}^nc_j+B\right\}, \quad for\quad  n\ge 0.
$$
\end{lema}

We will need the following auxiliary lemma for the error analysis of the robust case. The proof can be found in the appendix.

\begin{lema}[$G$-stability for first order finite differences]\label{lema_bosco}
Let $H$ be a Hilbert space with inner product $(\cdot,\cdot)_H$ and induced norm $\|\cdot\|_H$. 
Let $v^{n-q},\ldots, v^{n} \in H$ and denoting by 
$D v^n=v^n-v^{n-1}$ and so on, then for $q\in\{3,4,5\}$ there exists a constant $s_q>0$ and a symmetric positive definite matrix $G^q$ such that the following bound holds
\begin{equation}\label{cota_Aesta}
\Delta t \(\overline \partial_q v^n,D v^n\)_H\ge  s_q\left\|D v^n\right\|_{{H}}^2+\left\|D\bV^n\right\|_{G^q}^2-\left\|D\bV^{n-1}\right\|_{G^q}^2,
\end{equation}
where, according to \eqref{norma_G}
\[
\left|D\bV^n\right|_{G^q}^2=\sum_{i,j=1}^q g^q_{i,j}\left(Dv^{n-i+1},Dv^{n-j+1}\right)_H.
\]
\end{lema}

\section{Error analysis of the method. Non-robust case}\label{sec:5}
In this section we consider the fully discrete error analysis of the plain Galerkin method. To this end, we take
in \eqref{eq:bdfq_grad_div} $\mu=0$.
\subsection{Estimate for the velocity}

\begin{Theorem}[Non-robust error bound for the velocity]\label{thm:velo}
Assume that 
\begin{eqnarray*}\label{eq:regu}
\bu \in L^\infty(0,T;H^2(\Omega)^d) \cap l^2(0,T;H^{k+1}(\Omega)^d), \
\bu_t\in  H^{q+1}(0,T;H^1(\Omega)^d) \cap l^2(0,T;H^k(\Omega)^d),\nonumber\\
p \in L^\infty(0,T;H^1(\Omega)) \cap l^2(0,T;H^{k}(\Omega)),\
p_t \in H^{q+1}(0,T;L^2(\Omega))\cap l^2(0,T;H^{k-1}(\Omega)).\nonumber
\end{eqnarray*}
Let the time step be sufficiently small, where the concrete restrictions are given in \eqref{cond_1t}, \eqref{cond_t2}, 
and \eqref{cond_ht} below in the proof. In addition, it will be assumed the initial conditions are computed sufficiently accurately, concretely, 
that there exits a constant $C_{\mathrm{ic}}>0$ such that
\begin{equation}\label{ini}
\sum_{n=0}^{q-1} \|\bs_h^n-\bu_h^n\|_0^2\le C_{\mathrm{ic}} \(h^{2k+2}+(\Delta t)^{2q}\),
\end{equation}
where $\bs_h^n$ is the velocity component of the Stokes projection. Then the following 
estimate of the velocity error holds
\begin{equation}\label{eq:velo_est}
\left\|\bu(T) - \bu_h^M\right\|_0^2 \le
C \(h^{2k+2}+(\Delta t)^{2q}\),\ \Delta t \nu\sum_{n=q}^M\left\|\nabla \(\bu(t_n) - \bu_h^n\)\right\|_0^2\le
C \(h^{2k}+(\Delta t)^{2q}\),
\end{equation}
where $k$ is the degree of the piecewise polynomials of the velocity space and the  constant
on the right-hand side of \eqref{eq:velo_est} depends on $\nu^{-1}$ and on some norms of the velocity and pressure.  
\end{Theorem}

\begin{proof}
{\em Derivation of the error equation.}
As usual for deriving estimates for the velocity error, the discrete solution is compared with 
another discrete function for which estimates are available. To this end, 
we will compare $(\bu_h^n,p_h^n)$ with the  Stokes projection $(\bs_h^{n},p_h^{s,n}):=(\bs_h(t_n),p_h^{s}(t_n))$. Using \eqref{eq:gal_stokes} and \eqref{eq:stokes_rhs_g_1}, it is straightforward to show that for 
 $n\ge q$ and all $\bv_h\in V_h$
\begin{eqnarray}\label{eq:fullydiscrete}
\lefteqn{\(\overline \partial_q\bs_h^{n},\bv_h\)+\nu\(\nabla \bs_h^{n},\nabla \bv_h\)+b\(\bs_h^{n},\bs_h^{n},\bv_h\) -\(\nabla \cdot \bv_h, p_h^{s,n}\)
\nonumber}\\
&=& \(\boldsymbol f^{n},\bv_h\)+\(\overline \partial_q\bs_h^{n}-\bu_t^{n},\bv_h\)
+b\(\bs_h^{n},\bs_h^{n},\bv_h\)-b\(\bu^{n},\bu^{n},\bv_h\).
\end{eqnarray}
Then, denoting by 
\[
\be_h^n=\bs_h^n-\bu_h^n,
\]
and subtracting \eqref{eq:bdfq_grad_div} from \eqref{eq:fullydiscrete}, we get
for all $\bv_h\in V_h$  
\begin{eqnarray}\label{fully_error_gd}
\lefteqn{\hspace*{-12em}\(\overline \partial_q\be_h^{n},\bv_h\)+\nu\(\nabla \be_h^{n},\nabla \bv_h\)
+b\(\bs_h^{n}, \bs_h^{n},\bv_h\)-b\(\bu_h^n,\bu_h^n,\bv_h\)+\(\nabla \cdot \bv_h, p_h^n-p_h^{s,n}\)}
 \nonumber\\
 &=& \(\btau_1^{n}+\btau_2^{n},\bv_h\),
\end{eqnarray}
with
\begin{equation}\label{lostrun}
\btau_1^{n}=\overline \partial_q\bs_h^{n}-\bu_t^{n} \quad\mbox{and}\quad
\(\btau_2^{n},\bv_h\) = b\(\bs_h^{n},\bs_h^{n},\bv_h)-b(\bu^{n},\bu^{n},\bv_h\).
\end{equation}

{\em Application of $G$-stability.}
We choose in \eqref{fully_error_gd} $\bv_h=\be_h^n-\eta_q\be_h^{n-1}\in V_h^{\rm div}$ to get
\begin{eqnarray*}
\lefteqn{\(\overline \partial_q\be_h^{n},\be_h^n-\eta_q\be_h^{n-1}\)+\nu\|\nabla \be_h^{n}\|_0^2= \eta_q\nu \(\nabla\be_h^n,\nabla\be_h^{n-1}\)}\nonumber\\
&&+b\(\bu_h^n,\bu_h^n,\be_h^n-\eta_q\be_h^{n-1}\)-b\(\bs_h^{n}, \bs_h^{n},\be_h^n-\eta_q\be_h^{n-1}\)+\(\btau_1^{n}+\btau_2^{n},\be_h^n-\eta_q\be_h^{n-1}\).
\end{eqnarray*}
Let $\bE_h^n = [\be_h^n,\ldots,\be_h^{n-q+1}]$, using the definition of $\overline \partial_q$ from 
\eqref{eq_bdf_q}, and applying Lemma~\ref{lema_3_6} in combination with Lemma~\ref{lem:mult_tech} 
with $H=L^2(\Omega)^d$ yields
\begin{eqnarray}\label{eq:prin}
\lefteqn{\left|\bE_h^n\right|_G^2-\left|\bE_h^{n-1}\right|_G^2+\nu\Delta t \|\nabla \be_h^{n}\|_0^2
\le {\eta_q} \Delta t \nu(\nabla \be_h^{n},\nabla \be_h^{n-1})}\\
&& +\Delta t \left[b\(\bu_h^n,\bu_h^n,\be_h^n-\eta_q\be_h^{n-1}\)-b\(\bs_h^{n}, \bs_h^{n},\be_h^n-\eta_q\be_h^{n-1}\)\right]+\Delta t \(\btau_1^{n}+\btau_2^{n},\be_h^n-\eta_q\be_h^{n-1}\),\nonumber
\end{eqnarray}
where the definition \eqref{norma_G} of the $G$-norm was also used. 

{\em Estimate of nonlinear convective term.}
The nonlinear term on the right-hand side of \eqref{eq:prin} can be decomposed with respect 
to the test function 
\begin{eqnarray}\label{nonli1}
\lefteqn{b(\bu_h^n,\bu_h^n,\be_h^n-\eta_q\be_h^{n-1})-b(\bs_h^{n}, \bs_h^{n},\be_h^n-\eta_q\be_h^{n-1})}\\
&=&
\left[b\(\bu_h^n,\bu_h^n,\be_h^n\)-b\(\bs_h^{n}, \bs_h^{n},\be_h^n\)\right]
+\left[b\(\bu_h^n,\bu_h^n,-\eta_q\be_h^{n-1}\)-b\(\bs_h^{n}, \bs_h^{n},-\eta_q\be_h^{n-1}\)\right].\nonumber
\end{eqnarray}
Applying the skew-symmetric property of the trilinear term \eqref{eq:skew} together with  \eqref{cota_non_1} leads to 
\begin{eqnarray}\label{nonli2}
&&|b\(\bu_h^n,\bu_h^n,\be_h^n\)-b\(\bs_h^{n}, \bs_h^{n},\be_h^n\)|=|-b\(\be_h^n,\bs_h^{n},\be_h^n\)-b\(\bu_h^n,\be_h^n,\be_h^n\)|=|b\(\be_h^n,\be_h^n,\bs_h^{n}\)|\nonumber\\
&&\quad\le 2\| \bs_h^{n}\|_\infty\|\nabla\be_h^n\|_0\|\be_h^n\|_0
\le\frac{64}{\nu}\|\bs_h^{n}\|_\infty^2\|\be_h^n\|_0^2+\frac{\nu}{64}\|\nabla \be_h^n\|_0^2.
\end{eqnarray}
The second term on the right-hand side of \eqref{nonli1} can be rewritten in the form 
\begin{eqnarray}\label{nonli3med}
\lefteqn{\hspace*{-5em}b\(\bu_h^n,\bu_h^n,-\eta_q\be_h^{n-1}\)-b\(\bs_h^{n}, \bs_h^{n},-\eta_q\be_h^{n-1}\)=-b\(\be_h^n,\bs_h^{n},-\eta_q\be_h^{n-1}\)-b\(\bu_h^n,\be_h^n,-\eta_q\be_h^{n-1}\)}\nonumber\\
&=&b\(\be_h^n,\bs_h^{n},\eta_q\be_h^{n-1}\)-b\(\be_h^n,\be_h^n,\eta_q\be_h^{n-1}\)+b\(\bs_h^{n},\be_h^n,\eta_q\be_h^{n-1}\).
\end{eqnarray}
We now apply \eqref{cota_non_2}, \eqref{cota_non_1}, and \eqref{cota_non_3}  followed by Young's inequality
to bound the first, second, and third term on the right-hand side above
\begin{eqnarray}\label{nonli3}
&&\left|b\(\be_h^n,\bs_h^{n},\eta_q\be_h^{n-1}\)\right|
\le  \frac{32 \eta_q^2}{\nu} \(C \|\nabla\bs_h^{n}\|_{L^{2d/(d-1)}}^2+ \|\bs_h^{n}\|_{\infty}^2\)\|\be_h^{n-1}\|_0^2
+ \frac{\nu}{64}\|\nabla\be_h^n\|_0^2,\nonumber\\
&&\left|b\(\be_h^n,\be_h^{n},\eta_q\be_h^{n-1}\)\right|
\le \frac{64\eta_q^2}{\nu}\|\be_h^{n-1}\|_\infty^2\|\be_h^{n}\|_0^2+\frac{\nu}{64}\|\nabla \be_h^n\|_0^2,\\
&&\left|b\(\bs_h^{n},\be_h^n,\eta_q\be_h^{n-1}\)\right|
\le \frac{32\eta_q^2}{\nu}\left(\|\bs_h^n\|_{\infty}^2+C\|\nabla \cdot\bs_h^n\|_{L^{2d}}^2\right)\|\be_h^{n-1}\|_0^2
+\frac{\nu}{64}\|\nabla \be_h^n\|_0^2\nonumber.
\end{eqnarray} 
Summarizing, \eqref{nonli1}, \eqref{nonli2}, and \eqref{nonli3} gives 
\begin{eqnarray}\label{nonli_ult}
\lefteqn{b(\bu_h^n,\bu_h^n,\be_h^n-\eta_q\be_h^{n-1})-b(\bs_h^{n}, \bs_h^{n},\be_h^n-\eta_q\be_h^{n-1})
\le \frac{64}{\nu} \(\|\bs_h^{n}\|_\infty^2+\eta_q^2 \|\be_h^{n-1}\|_\infty^2 \)\|\be_h^n\|_0^2}\nonumber\\
&&+\frac{32\eta_q^2}{\nu}\left(2\|\bs_h^n\|_\infty^2+ C \|\nabla\bs_h^{n}\|_{L^{2d/(d-1)}}^2+ C\|\nabla\cdot\bs_h^n\|_{L^{2d}}^2\right)\|\be_h^{n-1}\|_0^2+\frac{\nu}{16}\|\nabla \be_h^n\|_0^2.
\end{eqnarray}

{\em Estimating remaining terms and summarize estimates.}
Going back to \eqref{eq:prin}, applying the Cauchy--Schwarz inequality and Young's inequality
yields for the first term on the right-hand side 
\begin{eqnarray}\label{eq:con}
\eta_q\nu \(\nabla\be_h^n,\nabla\be_h^{n-1}\)\le \frac{\eta_q}{2}\left(\nu\|\nabla\be_h^n\|_0^2+\nu\|\nabla\be_h^{n-1}\|_0^2\right).
\end{eqnarray}
The last term on the right-hand side of \eqref{eq:prin} is bounded by applying the dual pairing, 
the triangle and Young's inequality, so that 
\begin{eqnarray}\label{eq:trun}
\left|\(\btau_1^{n}+\btau_2^{n},\be_h^n-\eta_q\be_h^{n-1}\)\right|&\le& \|\btau_1^{n}+\btau_2^{n}\|_{-1}\(\|\nabla \be_h^n-\eta_q \nabla\be_h^{n-1}\|_0\)\nonumber\\
&\le&\frac{C}{\nu}\|\btau_1^{n}+\btau_2^{n}\|_{-1}^2+\frac{\nu}{32}\|\nabla \be_h^n\|_0^2+\frac{\nu}{32}\|\nabla\be_h^{n-1}\|_0^2,
\end{eqnarray}
where the constant $C$ depends on $\eta_q^2$.

Inserting \eqref{nonli_ult}, \eqref{eq:con}, and \eqref{eq:trun} in \eqref{eq:prin}, we reach
\begin{eqnarray}\label{eq:prin3}
\lefteqn{\left|\bE_h^n\right|_G^2-\left|\bE_h^{n-1}\right|_G^2+\frac{15\nu}{16}\Delta t \|\nabla \be_h^{n}\|_0^2-\Delta t \left(\frac{\eta_q}{2}+\frac{1}{32}\right)\left(\nu\|\nabla\be_h^n\|_0^2+\nu\|\nabla\be_h^{n-1}\|_0^2\right)}\nonumber\\
&\le& \Delta t \frac{64}{\nu}\(\|\bs_h^{n}\|_\infty^2+\eta_q^2 \|\be_h^{n-1}\|_\infty^2 \) \|\be_h^n\|_0^2+\frac{C\Delta t}{\nu}\|\btau_1^{n}+\btau_2^{n}\|_{-1}^2\nonumber\\
&& +\frac{32\eta_q^2}{\nu}\left(2\|\bs_h^n\|_\infty^2+ C \|\nabla\bs_h^{n}\|_{L^{2d/(d-1)}}^2+ C\|\nabla\cdot\bs_h^n\|_{L^{2d}}^2\right)\|\be_h^{n-1}\|_0^2.
\end{eqnarray}

Adding terms in \eqref{eq:prin3} from $n=q+1,\ldots,M$, we obtain
\begin{eqnarray}\label{eq:prin4}
\lefteqn{\left|\bE_h^M\right|_G^2+\frac{(14-16\eta_q)}{16}\Delta t \nu\sum_{n=q+1}^{M}\|\nabla \be_h^n\|_0^2\le 
\left|\bE_h^q\right|_G^2+\left(\frac{\eta_q}{2}+\frac{1}{32}\right)\Delta t \nu\|\nabla \be_h^q\|_0^2}\\
&&+ \Delta t c^q \|\be_h^q\|_0^2 +\Delta t\tilde{c}^M\|\be_h^M\|_0^2
+\Delta t\sum_{n=q+1}^{M-1} \(c^n+\tilde{c}^n\)\|\be_h^n\|_0^2+\frac{C}\nu\Delta t \sum_{n=q+1}^M
\|\btau_1^{n}+\btau_2^{n}\|_0^2,\nonumber
\end{eqnarray}
where, for $n=q,\ldots,M$, 
\begin{eqnarray}\label{constants}
c^n &=&\frac{32\eta_q^2}{\nu}\left(2\|\bs_h^n\|_\infty^2+ C \|\nabla\bs_h^{n}\|_{L^{2d/(d-1)}}^2+ C\|\nabla\cdot\bs_h^n\|_{L^{2d}}^2\right), \nonumber\\
\tilde{c}^n &=& \frac{64}{\nu}\(\|\bs_h^{n}\|_\infty^2+\eta_q^2 \|\be_h^{n-1}\|_\infty^2 \).
\end{eqnarray}
Let us fix a constant $c_\infty$ and let us assume for the moment that 
\begin{eqnarray}\label{apriori}
\|\be_h^n\|_\infty\le c_\infty,\quad n=0,1,\ldots,M-1.
\end{eqnarray}
We will prove \eqref{apriori} later by induction. 
The constants in \eqref{constants} can be bounded by 
taking into account \eqref{cotasinf}, \eqref{cotasl2d}, and  \eqref{cotas2d_dmenos1}
\begin{eqnarray}\label{constants_aco}
c^n&\le&c_{1,\bu}:=\frac{32\eta_q^2}{\nu}\left(8\ \esssup{0\le t\le T}\|\bu(t)\|_\infty^2 +C\ \esssup{0\le t\le T}\(\|\bu(t)\|_2+\|p(t)\|_1\)^2\right),\nonumber\\
\tilde{c}^n&\le&c_{2,\bu}:= \frac{64}{\nu}\(4\esssup{0\le t\le T}\|\bu(t)\|_\infty^2+ \eta_q^2c_\infty^2\).
\end{eqnarray}
From \eqref{eta}, 
we observe that ${(14-16\eta_q)}\ge 1,$ $q=1,\ldots,5$.
Applying \eqref{equi_nor} to \eqref{eq:prin4} gives
\begin{eqnarray}\label{eq:prin5}
\lefteqn{\lambda_{\rm min}\|\be_h^M\|_0^2+\frac{\Delta t}{16} \nu\sum_{n=q+1}^M\|\nabla \be_h^n\|_0^2\le\lambda_{\rm max}\sum_{n=1}^{q-1} \|\be_h^n\|_0^2+c^*\Delta t \nu\|\nabla \be_h^q\|_0^2+\hat{c}  \|\be_h^q\|_0^2}\nonumber\\
&& +\Delta t{c_{2,\bu}}\|\be_h^M\|_0^2+\Delta t\sum_{n=q+1}^{M-1} \(c_{1,\bu}+c_{2,\bu}\)\|\be_h^n\|_0^2+\Delta t\frac{C}\nu \sum_{n=q+1}^M
\|\btau_1^{n}+\btau_2^{n}\|_0^2,
\end{eqnarray}
where $\hat{c}=\lambda_{\rm max}+\Delta t c^q\le \lambda_{\rm max}+{\Delta t}c_{1,\bu},$ and $c^*=\left(\nicefrac{\eta_q}{2}+\nicefrac{1}{32}\right)$.

{\em Deriving a form of the estimate that fits the discrete Gronwall lemma.}
We need to bound the second and third terms on the right-hand side of \eqref{eq:prin5}. To this end, we consider \eqref{fully_error_gd}
for $n=q$ and $\bv_h=\be_h^q$, use \eqref{eq_bdf_q}, apply the dual pairing, and the Cauchy--Schwarz
inequality to get 
\begin{eqnarray}\label{nq}
\delta_0 \|\be_h^q\|_0^2+\Delta t \nu\|\nabla \be_h^{q}\|_0^2&\le&
\Delta t\left|(b\(\bs_h^{q}, \bs_h^{q},\be_h^q\)-b\(\bu_h^q,\bu_h^q,\be_h^q\)\right|+\Delta t\|\btau_1^{q}+\btau_2^{q}\|_{-1}\|\nabla \be_h^q\|_0\nonumber\\
&&+\left\|\sum_{i=1}^q\delta_i\be_h^{q-i}\right\|_0\|\be_h^q\|_0.
\end{eqnarray}
Utilizing \eqref{nonli2} with $n=q$ and Young's inequality yields
\begin{equation}\label{ehq_1}
\frac{3\delta_0}4 \|\be_h^q\|_0^2+\Delta t \frac{31}{32}\nu\|\nabla \be_h^{q}\|_0^2 \le
 \Delta t\frac{64}{\nu}\|\bs_h^{q}\|_\infty^2\|\be_h^q\|_0^2+\Delta t\frac{C}{\nu}\|\btau_1^{q}+\btau_2^{q}\|_{-1}^2 +\frac{1}{\delta_0}\left\|\sum_{i=1}^q\delta_i\be_h^{q-i}\right\|_0^2.
\end{equation}
In view of \eqref{cotasinf}, we assume
\begin{eqnarray}\label{cond_1t}
\Delta t <\frac{\delta_0}{4}\left(\frac{128}{\nu}\esssup{0\le t\le T}\|\bu(t)\|_\infty^2 \right)^{-1},
\end{eqnarray}
so that we get from \eqref{ehq_1}
\begin{eqnarray}\label{ehq_2}
\frac{\delta_0}{2} \|\be_h^q\|_0^2+\Delta t \frac{31}{32}\nu\|\nabla \be_h^{q}\|_0^2\le
\Delta t\frac{C}{\nu} \|\btau_1^{q}+\btau_2^{q}\|_{-1}^2 +\frac{1}{\delta_0}\left\|\sum_{i=1}^q\delta_i\be_h^{q-i}\right\|_0^2.
\end{eqnarray}
From \eqref{ehq_2} we can infer that there exists a constant $C$ such that
\begin{equation}\label{cond_iniq}
c^*\Delta t \nu\|\nabla \be_h^q\|_0^2+\hat{c}_q  \|\be_h^q\|_0^2\le {\Delta t} \frac{C}{\nu}\|\btau_1^{q}+\btau_2^{q}\|_{-1}^2 +C\sum_{n=0}^{q-1}\|\be_h^{n}\|_0^2.
\end{equation}
Inserting \eqref{cond_iniq} into \eqref{eq:prin5} we obtain, after having reordered terms, 
\begin{eqnarray}\label{eq:prindef}
\lefteqn{\|\be_h^M\|_0^2+\frac{\Delta t \nu}{16\lambda_{\rm min}}\sum_{n=q+1}^M\|\nabla \be_h^n\|_0^2\le \Delta t\frac{c_{2,\bu}}{\lambda_{\rm min}}\|\be_h^M\|_0^2} \nonumber\\
&&+\Delta t\sum_{n=q+1}^{M-1} \frac{(c_{1,\bu}+c_{2,\bu})}{\lambda_{\rm min}}\|\be_h^n\|_0^2+\Delta t \frac{C}{\nu}\sum_{n=q}^M\|\btau_1^{n}+\btau_2^{n}\|_{-1}^2+ C\sum_{n=0}^{q-1} \|\be_h^n\|_0^2.
\end{eqnarray}

{\em Application of the discrete Gronwall lemma.} Estimate \eqref{eq:prindef} is in such a form that the discrete Gronwall lemma, Lemma~\ref{gronwall2},
can be applied. To this end, take $a_n=\|\be_h^n\|_0^2$, $b_n=\Delta t \nu\|\nabla \be_h^n\|_0^2/({16\lambda_{\rm min}})$, 
$\gamma_n=\Delta t c_{2,\bu}/\lambda_{\rm min}$, $\delta_n=\Delta t c_{1,\bu}/\lambda_{\rm min}$
and the last two terms in \eqref{eq:prindef} represent the last two terms in \eqref{eq:gronwall1}.
We can apply Lemma~\ref{gronwall2} if
$
\gamma_n<1.
$
From now on, we will assume
\begin{eqnarray}\label{cond_t2}
 \frac{\Delta t}{\lambda_{\rm min}} c_{2,\bu}\le \frac{1}{2},
\end{eqnarray}
so that $\gamma_n\le 1/2$ and $\sigma_n=(1-\gamma_n)^{-1}\le 2$, $n=q+1,\ldots,M$. 
Then, the application of the discrete Gronwall lemma yields
\begin{eqnarray*}
\lefteqn{\|\be_h^M\|_0^2+\frac{\Delta t \nu}{16\lambda_{\rm min}}\sum_{n=q+1}^M\|\nabla \be_h^n\|_0^2}\\
&\le& C \exp\left(\sigma_M\gamma_M+\sum_{n=q+1}^{M-1}(\sigma_j\gamma_j+\delta_j)\right)\left(\frac{\Delta t}{\nu}\sum_{n=q}^M\|\btau_1^{n}+\btau_2^{n}\|_{-1}^2+ \sum_{n=0}^{q-1} \|\be_h^n\|_0^2\right)\nonumber\\
&\le& C \ee^{T(c_{1,\bu}+2c_{2,\bu})/\lambda_{\rm min}}\left(\frac{\Delta t}{\nu}\sum_{n=q}^M\|\btau_1^{n}+\btau_2^{n}\|_{-1}^2+ \sum_{n=0}^{q-1} \|\be_h^n\|_0^2\right).
\end{eqnarray*}
Adding a term on the left-hand side, 
taking into account \eqref{cond_iniq}, and simplifying the representation of the constants, then 
there is a constant $C>0$ such that 
\begin{equation}\label{eq:coner}
\|\be_h^M\|_0^2+\Delta t \nu\sum_{n=q}^M\|\nabla \be_h^n\|_0^2\le C  \ee^{T(c_{1,\bu}+2c_{2,\bu})/\lambda_{\rm min}} \left(\frac{\Delta t}{\nu}\sum_{n=q}^M\|\btau_1^{n}+\btau_2^{n}\|_{-1}^2+ \sum_{n=0}^{q-1} \|\be_h^n\|_0^2\right).
\end{equation}

{\em Estimate terms on the right-hand side of \eqref{eq:coner}.}
The term with the initial conditions is bounded by assumption \eqref{ini}.

We will bound now the truncation errors. 
Recalling \eqref{lostrun}, for $\btau_1^n$ we have
\begin{equation}\label{cota_trun_1}
\|\btau_1^{n}\|_{-1}=\|\overline \partial_q\bs_h^{n}-\bu_t^{n}\|_{-1}\le \|\overline \partial_q\bs_h^{n}-\bs_t^{n}\|_{-1}
+\|\bs_t^{n}-\bu^n_t\|_{-1}.
\end{equation}
The second term on the right-hand side can be bounded by utilizing \eqref{sh_menos1}, 
which gives 
\begin{equation}\label{cota_trun_2}
\|\bs_t^{n}-\bu^n_t\|_{-1}\le C h^{k+1}\(\|\bu_t^n\|_k+\|p_t^n\|_{k-1}\).
\end{equation}
To bound the first term, we apply \cite[(B.3)]{bdf_stokes}
\begin{equation}\label{cota_trun_3}
\|\overline \partial_q\bs_h^{n}-\bs_t^{n}\|_{-1}\le \|\overline \partial_q\bs_h^{n}-\bs_t^{n}\|_{0}\le C(\Delta t)^{q-1/2}\left(\int_{t_{n-q}}^{t_n}\|\partial_t^{q+1}\bs_h(s)\|_0^2 \ \di s\right)^{1/2},
\end{equation}
for a positive constant $C$.
Then, from \eqref{cota_trun_1}, \eqref{cota_trun_2}, and \eqref{cota_trun_3}
we get
\begin{equation}\label{tiempo2}
\sum_{n=q}^M\Delta t \|\btau_1^n\|_{-1}^2\le C h^{2k+2}\Delta t\sum_{n=q}^M\(\|\bu_t^n\|_k^2+\|p_t^n\|_{k-1}^2\)
+C (\Delta t)^{2q}\int_0^T\|\partial_t^{q+1}\bs_h(s)\|_0^2 \ \di s.
\end{equation}
Let us observe that, applying {\eqref{eq:cotastokes}}
we can bound
\begin{eqnarray}\label{er_pro_u}
\int_0^T\|\partial_t^{q+1}\bs_h(s)\|_0^2 \ \di s&\le &2\int_0^T\|\partial_t^{q+1}\bu(s)\|_0^2 \ \di s
+2\int_0^T\|\partial_t^{q+1}(\bs_h-\bu(s))\|_0^2 \ \di s\nonumber\\
&\le &2\int_0^T\|\partial_t^{q+1}\bu(s)\|_0^2 \ \di s
+Ch\int_0^T\(\|\partial_t^{q+1} \bu(s)\|_1^2 +\|\partial_t^{q+1}p(s)\|_0^2\)\ \di s\nonumber\\
&\le &C\int_0^T\|\partial_t^{q+1}\bu(s)\|_0^2 \ \di s,
\end{eqnarray}
for $h$ small enough.
For $\btau_2^n$, recalling \eqref{lostrun},  we have
\begin{equation}\label{cota_trun_21}
\|\btau_2^{n}\|_{-1}=\sup_{\bv\in V, \bv\neq \boldsymbol 0}\frac{\left|b\(\bs_h^{n},\bs_h^{n},\bv\)-b\(\bu^{n},\bu^{n},\bv\)\right|}{\|\nabla \bv\|_0}.
\end{equation}
A straightforward calculation leads to 
\begin{eqnarray}\label{asin2}
b(\bs_h^{n},\bs_h^{n},\bv)-b(\bu^{n},\bu^{n},\bv)=b(\bs_h^{n}-\bu^n,\bs_h^{n},\bv)+b(\bu^n,\bs_h^{n}-\bu^n,\bv).
\end{eqnarray}
For the first term on the right-hand side of \eqref{asin2}, \eqref{cota_non_4}, and \eqref{eq:cotastokes} are  used to obtain
\begin{eqnarray}\label{asin3}
b(\bs_h^{n}-\bu^n,\bs_h^{n},\bv)
&\le& C\(\|\nabla \bs_h^{n}\|_{L^{2d}}+\|\bs_h^n\|_\infty\)\|\bs_h^{n}-\bu^n\|_0 \|\nabla \bv\|_0\nonumber\\
&\le& C h^{k+1} \(\|\nabla \bs_h^{n}\|_{L^{2d}}+\|\bs_h^n\|_\infty\)\(\|\bu^n \|_{k+1}+\|p^n\|_k\)\|\nabla \bv\|_0.
\end{eqnarray}
Let us observe that $\|\bs_h^n\|_\infty$ and $\|\nabla \bs_h^{n}\|_{L^{2d}}$ are bounded in \eqref{cotasinf} and \eqref{cotas2d}, respectively.
For the second term on the right-hand side of \eqref{asin2}, using \eqref{eq:skew}, taking into account that $\nabla \cdot \bu^n=0$,
and applying \eqref{eq:cotastokes}, we get
\begin{eqnarray}\label{asin4}
b\(\bu^n,\bs_h^{n}-\bu^n,\bv\)&=&b\(\bu^n,\bv,\bs_h^{n}-\bu^n\)\le \|\bu^n\|_\infty\|\nabla\bv\|_0\|\bs_h^{n}-\bu^n\|_0
\nonumber\\
&\le&Ch^{k+1}\|\bu^n\|_\infty \(\|\bu^n \|_{k+1}+\|p^n\|_k\)\|\nabla \bv\|_0.
\end{eqnarray}
With the Sobolev embedding \eqref{sob1} we can bound $\|\bu^n\|_\infty\le C\|\bu^n\|_2$. 
Then, from \eqref{cota_trun_21}--\eqref{asin4}, we obtain
\begin{eqnarray*}
\|\btau_2^n\|_{-1}&\le& C h^{k+1} \(\|\bu^n\|_2+\|p^n\|_1\)\(\|\bu^n \|_{k+1}+\|p^n\|_k\)\nonumber\\
&\le& Ch^{k+1}\esssup{0\le t\le T}\(\|\bu(t)\|_2+\|p(t)\|_1\)\(\|\bu^n \|_{k+1}+\|p^n\|_k\),
\end{eqnarray*}
so that
\begin{eqnarray}\label{trunca2}
\sum_{n=q}^M\Delta t \|\btau_2^n\|_{-1}^2\le C h^{2k+2}\Delta t\sum_{n=q}^M\(\|\bu^n\|_{k+1}^2+\|p^n\|_{k}^2\).
\end{eqnarray}

Inserting \eqref{ini}, \eqref{tiempo2}, and \eqref{trunca2} in \eqref{eq:coner} we get for a constant $C_{\mathrm{e}}>0$
\begin{eqnarray}\label{eq:cota_fin}
\lefteqn{\|\be_h^M\|_0^2+\Delta t \nu\sum_{n=q}^M\|\nabla \be_h^n\|_0^2\le C  \ee^{T(c_{1,\bu}+2c_{2,\bu})/\lambda_{\rm min}} \left[h^{2k+2}+(\Delta t)^{2q}
\right.}\nonumber\\
&&+ h^{2k+2}\sum_{n=q}^M\Delta t\(\|\bu_t^n\|_k^2+\|p_t^n\|_{k-1}^2\)
+(\Delta t)^{2q}\int_0^T\|\partial_t^{q+1}\bu(s)\|_0^2 \ \di s\nonumber\\
&&\left. h^{2k+2}\sum_{n=q}^M\Delta t\(\|\bu^n\|_{k+1}^2+\|p^n\|_{k}^2\)\right]\le C_{\mathrm{e}} \(h^{2k+2}+(\Delta t)^{2q}\).
\end{eqnarray}
Notice that from \eqref{eq:cota_fin} one gets super-convergence in space with respect to the norm of 
$V$ between the Galerkin approximation and the Stokes projection, in agreement with \cite{paper_tesis,paper_NS_lin}, where this super-convergence result was proved for the 
continuous-in-time case.

{\em Proof of assumption \eqref{apriori}.} The derivation of \eqref{eq:cota_fin}
used the assumption \eqref{apriori}.
Then, we need to prove that for $j=0,1,\ldots,M-1$, the condition
\begin{equation}\label{apri}
\|\be_h^j\|_\infty\le c_\infty
\end{equation}
holds.
We will prove this by induction whenever the time step $\Delta t$ and the mesh size $h$ are small enough. Let us fix a, possibly large,  positive constant $c_\infty$. Assume that the mesh width, and after having set 
the mesh width also the time step, are sufficiently small such that 
\begin{equation}\label{cond_ht}
\begin{array}{rclrcl}
c_{\rm inv} C_{\mathrm{ic}}^{1/2} h^{k+1-d/2}&\le& \displaystyle \frac{c_\infty}{2}, &  c_{\rm inv}C_{\mathrm{e}}^{1/2}h^{k+1-d/2}&\le& \displaystyle \frac{c_\infty}{2},\\[1em]
c_{\rm inv} C_{\mathrm{ic}}^{1/2} h^{-d/2}(\Delta t)^q&\le& \displaystyle \frac{c_\infty}{2},& c_{\rm inv}C_{\mathrm{e}}^{1/2}h^{-d/2}(\Delta t)^q&\le& \displaystyle \frac{c_\infty}{2}.
\end{array}
\end{equation}

Consider first the initial values. Using the inverse inequality \eqref{inv} gives
\[
\|\be_h^n\|_\infty\le c_{\rm inv} h^{-d/2}\|\be_h^n\|_0.
\]
For $j=0,\ldots,q-1$, applying \eqref{ini} and \eqref{cond_ht}, we get
\[
\|\be_h^j\|_\infty\le c_{\rm inv} h^{-d/2}\|\be_h^j\|_0\le c_{\rm inv} C_{\mathrm{ic}}^{1/2}h^{-d/2}\(h^{k+1}+(\Delta t)^q\)\le c_\infty,
\]
so that condition \eqref{apriori} is satisfied for $n=0,\ldots,q-1$.

Assume now \eqref{apri} holds for $j\le n-1$.
Then, we can apply the arguments that gave \eqref{eq:cota_fin} for $n$ instead of $M$ so that
\[
\|\be_h^n\|_0\le C_{\mathrm{e}}^{1/2}\left(h^{k+1}+(\Delta t)^q\right).
\]
Applying now \eqref{cond_ht} yields
\[
\|\be_h^n\|_\infty\le c_{\rm inv} h^{-d/2}\|\be_h^n\|_0\le c_{\rm inv} h^{-d/2}C_{\mathrm{e}}^{1/2}\left(h^{k+1}+(\Delta t)^q\right)\le c_\infty,
\]
so that \eqref{apri} holds for $n$. As a consequence, \eqref{apriori} is true and the bound \eqref{eq:cota_fin} holds for the 
error bound of the fully discrete method.

{\em Application of the triangle inequality.} The proof finishes by applying the triangle
inequality $\|\bu(T) - \bu_h^M\|_0 \le \|\bu(T)-\bs_h^n\|_0  +\|\be_h^M\|_0$ and 
analogously to the $L^2(\Omega)$ norm of the gradient and then applying the estimate 
\eqref{eq:cotastokes}.
\end{proof}
\begin{remark}\label{remark1}
To derive the velocity error estimate given in Theorem~\ref{thm:velo}, 
we assumed conditions \eqref{cond_1t}, \eqref{cond_t2}, 
and \eqref{cond_ht}, which can be interpreted as follows. First, the time step should be small enough, where the smallness depends in essence on data of the problem and norms of its solution (\eqref{cond_1t}, \eqref{cond_t2}). Second, the quantity
$
h^{-d/2}(\Delta t)^q
$
should remain bounded, which gives a restriction of type $\Delta t \le C h^{d/(2q)}$.
This is a quite mild restriction for $q\ge3$. Notice that in the case $q\le 2$, resp. $\eta_q=0$ (BDF-1 and BDF-2 methods are $A$-stable), this restriction is not needed, compare \eqref{constants_aco}.
\end{remark}

\subsection{Error bound for the pressure}

\begin{Theorem}[Error bound for the pressure] \label{thpre1} Let the assumptions of Theorem~\ref{thm:velo}, assumption \eqref{con_ini_norm1} below for the initial conditions, and the time step restriction \eqref{cota_t_pre} be satisfied, then  
\[
\Delta t \sum_{n=q}^M \|p(t_n)-p_h^n\|_0^2\le C\(h^{2k}+(\Delta t)^{2q}\).
\]
The constant depends on $\nu^{-1}$ and on the inverse of the inf-sup constant. 
\end{Theorem}

\begin{proof}
{\em Applying the discrete inf-sup condition.} Proving a bound for the pressure starts with the decomposition 
\[
p_n-p_h^n =\(p_n-p_h^{s,n}\)  -\(p_h^n-p_h^{s,n}\).
\]
Then, using the discrete inf-sup stability \eqref{LBB}, the error equation \eqref{fully_error_gd},
the dual pairing, and the Cauchy--Schwarz inequality, 
we get for the second term 
\begin{align}\label{co_pre1}
\|p_h^n-p_h^{s,n}\|_0&\le \frac{1}{\beta_0}\left(\|\overline \partial_q\be_h^n\|_{-1}+\nu \|\nabla \be_h^n\|_0+
\sup_{\bv_h\in V_h,\bv_h\neq \boldsymbol 0}\frac{\left|b\(\bs_h^n,\bs_h^n,\bv_h\)-b\(\bu_h^n,\bu_h^n,\bv_h\)\right|}{\|\nabla\bv_h\|_0}\right)\nonumber\\
 &\quad+\frac{1}{\beta_0}\|\btau_1^n+\btau_2^n\|_{-1}.
\end{align}
This relation is squared,  multiplied by $\Delta t$, and then a summation over the time instants is performed. 
For the second term on the right-hand side, we obtain by using \eqref{eq:cota_fin}
\begin{equation}\label{eq:press_02}
\Delta t \nu \sum_{n=q}^M \|\nabla \be_h^n\|_0^2 \le C_{\mathrm{e}} \(h^{2k+2}+(\Delta t)^{2q}\).
\end{equation}
The third term is estimated as performed in \eqref{asin2}--\eqref{asin3}, with $\bu^n$ replaced by $\bu_h^n$. 
In \eqref{asin3} we apply \eqref{cotasinf} and \eqref{cotas2d}. Then, instead
of \eqref{asin4}, we apply \eqref{cota_non_5} to get
\[
b\(\bu_h^n,\bs_h^{n}-\bu_h^n,\bv_h\)\le \|\nabla \bu_h^n\|_0\|\nabla \be_h^n\|_0\|\nabla \bv_h\|_0
\le \(\|\nabla \be_h^n\|_0+\|\nabla\bs_h^n\|_0\)\|\nabla\be_h^n\|_0\|\nabla \bv_h\|_0.
\]
To bound $\|\nabla\bs_h^n\|_0$ we observe that
adding and subtracting $\nabla \bu^n$ and applying \eqref{eq:cotastokes} gives $\|\bs_h^n\|_1\le C(\|\bu^n\|_1+\|p^n\|_0)$. Then, 
we finally obtain, with Poincar\'e's inequality \eqref{eq:poin} and taking into account that $\|\nabla \be_h^n\|_0$ is bounded, which follows from \eqref{eq:cota_fin}, the inverse inequality \eqref{inv}, and \eqref{cond_ht},
\begin{eqnarray*}
\sup_{\bv_h\in V_h,\bv_h\neq \boldsymbol 0}\frac{\left|b\(\bs_h^n,\bs_h^n,\bv_h\)-b\(\bu_h^n,\bu_h^n,\bv_h\)\right|}{\|\nabla\bv_h\|_0} \le
 C\(\|\be_h^n\|_0+\|\nabla \be_h^n\|_0^{2}+\|\nabla \be_h^n\|_0\)\le C\|\nabla \be_h^n\|_0,
\end{eqnarray*}
so that the bound of this term concludes applying \eqref{eq:press_02} (with a constant depending in $\nu^{-1}$).
Using \eqref{tiempo2} and \eqref{trunca2} leads for the fourth term to 
\begin{equation}\label{eq:press_04}
\Delta t \sum_{n=q}^M \|\btau_1^n+\btau_2^n\|_{-1}^2 \le C \(h^{2k+2}+(\Delta t)^{2q}\).
\end{equation}
 The only term in \eqref{co_pre1} that is not yet bounded is the first one. Instead of $\|\overline \partial_q\be_h^n\|_{-1}$,
we are going to bound $\|\overline \partial_q\be_h^n\|_{0}$ or, more precisely,
$
\sum_{n=q}^M \Delta t \|\overline \partial_q\be_h^n\|_{0}^2,
$
which, as a consequence, will give the bound for
$
\sum_{n=q}^M \Delta t \|p_h^n-p_h^{s,n}\|_0^2.
$

{\em First step of bounding the temporal derivative of the velocity difference, application of $G$-stability.}
To get the bound for the time derivative of the error we argue as in \cite[Theorem 4.2]{bdf_stokes}. We first take in \eqref{fully_error_gd} $\bv_h=\overline \partial_q\be_h^n \in V_h^{\rm div}$ to obtain
\begin{equation}\label{eq;derit1}
\(\overline \partial_q\be_h^n,\overline \partial_q\be_h^n\)+\nu\(\nabla \be_h^n,\nabla \overline \partial_q\be_h^n\)
+b\(\bs_h^n,\bs_h^n,\overline \partial_q\be_h^n\)-b\(\bu_h^n,\bu_h^n,\overline \partial_q\be_h^n\)=\(\btau_1^n+\btau_2^n,\overline \partial_q\be_h^n\).
\end{equation}
Then, we consider equation \eqref{fully_error_gd} for $n-1$ instead of $n$ and take $\bv_h=-\eta_q\overline \partial_q\be_h^n\in V_h^{\rm div}$, which gives
\begin{eqnarray}\label{eq;derit2}
\lefteqn{\hspace*{-13em}-\eta_q\(\overline \partial_q\be_h^{n-1},\overline \partial_q\be_h^n\)-\eta_q\nu\(\nabla \be_h^{n-1},\nabla \overline \partial_q\be_h^n\)
-\eta_q b\(\bs_h^{n-1},\bs_h^{n-1},\overline \partial_q\be_h^n\)}\nonumber\\
+\eta_q b\(\bu_h^{n-1},\bu_h^{n-1},\overline \partial_q\be_h^n\) &=&-\eta_q\(\btau_1^{n-1}+\btau_2^{n-1},\overline \partial_q\be_h^n\).
\end{eqnarray}
Adding \eqref{eq;derit1} and \eqref{eq;derit2} and reordering terms leads to
\begin{eqnarray}\label{eq;derit3}
\lefteqn{\nu\(\nabla \overline \partial_q\be_h^n,\nabla \(\be_h^n-\eta_q \be_h^{n-1}\)\)+\(\overline \partial_q\be_h^n,\overline \partial_q\be_h^n-\eta_q\overline \partial_q\be_h^{n-1}\)+b\(\bs_h^{n},\bs_h^{n},\overline \partial_q\be_h^n\)}\nonumber \\
&&-b\(\bu_h^{n},\bu_h^{n},\overline \partial_q\be_h^n\)-\eta_q b\(\bs_h^{n-1},\bs_h^{n-1},\overline \partial_q\be_h^n\)+\eta_q b\(\bu_h^{n-1},\bu_h^{n-1},\overline \partial_q\be_h^n\) \nonumber\\
&=&\(\btau_1^n+\btau_2^n-\eta_q\(\btau_1^{n-1}+\btau_2^{n-1}\),\overline \partial_q\be_h^n\).
\end{eqnarray}

Let us denote by
\[
\left|\bV^n\right|_{G,1}^2=\sum_{i,j=1}^q g_{i,j}\left(\nabla v^{n-i+1},\nabla v^{n-j+1}\right).
\]
Applying Lemma~\ref{lema_3_6} in combination with Lemma~\ref{lem:mult_tech}
with $H=V$ yields
\begin{eqnarray}\label{eq;derit4}
\lefteqn{\nu \left|\bE^n\right|_{G,1}^2-\nu \left|\bE^{n-1}\right|_{G,1}^2+\Delta t \|\overline \partial_q\be_h^n\|_0^2}\nonumber \\
&\le& \eta_q\Delta t\(\overline \partial_q\be_h^n,\overline \partial_q\be_h^{n-1}\)
 -\Delta t \left[b\(\bs_h^{n},\bs_h^{n},\overline \partial_q\be_h^n\)-b\(\bu_h^{n},\bu_h^{n},\overline \partial_q\be_h^n\)\right]\nonumber\\
&& +\eta_q \Delta t \left[b\(\bs_h^{n-1},\bs_h^{n-1},\overline \partial_q\be_h^n\)- b\(\bu_h^{n-1},\bu_h^{n-1},\overline \partial_q\be_h^n\)\right]
\nonumber\\
&& +\Delta t \(\btau_1^n+\btau_2^n-\eta_q\(\btau_1^{n-1}+\btau_2^{n-1}\),\overline \partial_q\be_h^n\).
\end{eqnarray}
To bound the right-hand side of \eqref{eq;derit4}, we can utilize the same techniques as before. 
Applying the Cauchy--Schwarz and the Young inequality gives for the first term 
\begin{equation}\label{aux1}
\eta_q\Delta t\(\overline \partial_q\be_h^n,\overline \partial_q\be_h^{n-1}\)\le \frac{\eta_q}{2}\Delta t \left(\|\overline \partial_q\be_h^n\|_0^2+\|\overline \partial_q\be_h^{n-1}\|_0^2\right).
\end{equation}
Using the decomposition \eqref{asin2} together with \eqref{cota_non_2}, \eqref{cota_non_3} gives
\begin{eqnarray*}
\lefteqn{\left|b\(\bs_h^{n},\bs_h^{n},\overline \partial_q\be_h^n\)-b\(\bu_h^{n},\bu_h^{n},\overline \partial_q\be_h^n\)\right|}
\nonumber\\
&\le& C\(\|\nabla \bs_h^n\|_{L^{2d/(d-1)}}+\|\bs_h^n\|_\infty+\|\bu_h^n\|_\infty+\|\nabla \cdot \bu_h^n\|_{L^{2d/(d-1)}}\) \|\nabla(\bs_h^n-\bu_h^n)\|_0\|\overline \partial_q\be_h^n\|_0.
\end{eqnarray*}
We observe that $\|\nabla \bs_h^n\|_{L^{2d/(d-1)}}$, $\|\bs_h^n\|_\infty$ are bounded in \eqref{cotas2d_dmenos1} and \eqref{cotasinf}. The third term in the parentheses can be bounded by applying \eqref{inv}, \eqref{eq:cota_fin}, \eqref{cond_ht}, and \eqref{cotasinf}
\begin{eqnarray*}
\|\bu_h^n\|_\infty\le \|\be_h^n\|_\infty+\|\bs_h^n\|_\infty&\le& c_{\rm inv} h^{-d/2}\|\be_h^n\|_0+\|\bs_h^n\|_\infty
\\&\le& c_{\rm inv} h^{-d/2}C_{\mathrm{e}}^{1/2} \(h^{k+1}+(\Delta t)^q\)+\|\bs_h^n\|_\infty\\
&\le& c_{\infty}+2\|\bu^n\|_\infty\le c_{\infty}+2 \esssup{0\le t\le T}\|\bu(t)\|_\infty.
\end{eqnarray*}
Applying \eqref{inv}, \eqref{eq:cota_fin} and \eqref{cotas2d_dmenos1} to the last term in parentheses yields
\begin{eqnarray*}
\|\nabla \cdot \bu_h^n\|_{L^{2d/(d-1)}} &\le &\|\nabla \cdot \be_h^n\|_{L^{2d/(d-1)}}+\|\nabla \bs_h^n\|_{L^{2d/(d-1)}}\\
&\le& c_{\mathrm{inv}}h^{{-3/2}}\|\be_h^n\|_{0}+\|\nabla  \bs_h^n\|_{L^{2d/(d-1)}}\nonumber\\
&\le& c_{\mathrm{inv}}h^{{-3/2}} C_{\mathrm{e}}^{1/2}\(h^{k+1}+(\Delta t)^q\)+C\esssup{0\le t\le T}\(\|\bu(t)\|_{2}+\|p(t)\|_1\),
\end{eqnarray*}
which is bounded whenever
\begin{equation}\label{cota_t_pre}
c_{\mathrm{inv}} C_{\mathrm{e}}^{1/2}(\Delta t)^q\le h^{{3/2}}.
\end{equation}
Assuming \eqref{cota_t_pre} gives the following bound for the second term on the right-hand side of 
\eqref{eq;derit4}
\begin{eqnarray}\label{aux2}
\left|b\(\bs_h^{n},\bs_h^{n},\overline \partial_q\be_h^n\)-b\(\bu_h^{n},\bu_h^{n},\overline \partial_q\be_h^n\)\right|
\le C \|\nabla \be_h^n\|_0\|\overline \partial_q\be_h^n\|_0\le C \|\nabla \be_h^n\|_0^2+\frac{1}{48}\|\overline \partial_q\be_h^n\|_0^2.
\end{eqnarray}
The same argument applied to the third term on the right-hand side of \eqref{eq;derit4}  gives
\begin{eqnarray}\label{aux3}
\eta_q \left|b\(\bs_h^{n-1},\bs_h^{n-1},\overline \partial_q\be_h^n\)- b\(\bu_h^{n-1},\bu_h^{n-1},\overline \partial_q\be_h^n\)\right|
\le C \|\nabla \be_h^{n-1}\|_0^2+\frac{1}{48}\|\overline \partial_q\be_h^n\|_0^2.
\end{eqnarray}
And, finally, for the last  term on the right-hand side of \eqref{eq;derit4}, we get
\begin{eqnarray}\label{aux4}
\lefteqn{
\left|\(\btau_1^n+\btau_2^n-\eta_q\(\btau_1^{n-1}+\btau_2^{n-1}\),\overline \partial_q\be_h^n\)\right|}\nonumber\\
&\le& C\(\|\btau_1^n+\btau_2^n\|_0^2+
\|\btau_1^{n-1}+\btau_2^{n-1}\|_0^2\)+\frac{1}{48}\|\overline \partial_q\be_h^n\|_0^2.
\end{eqnarray}
Inserting \eqref{aux1}, \eqref{aux2}, \eqref{aux3}, and \eqref{aux4} in \eqref{eq;derit4} yields
\begin{eqnarray}\label{eq;derit5}
\lefteqn{\nu \left|\bE^n\right|_{G,1}^2-\nu \left|\bE^{n-1}\right|_{G,1}^2+\frac{15}{16}\Delta t \|\overline \partial_q\be_h^n\|_0^2
-\frac{\eta_q}{2}\Delta t \|\overline \partial_q\be_h^n\|_0^2-
\frac{\eta_q}{2}\Delta t\|\overline \partial_q\be_h^{n-1}\|_0^2}\nonumber\\
&&\le C\Delta t\left(\|\nabla \be_h^n\|_0^2+\|\nabla \be_h^{n-1}\|_0^2+\|\btau_1^n+\btau_2^n\|_0^2+
\|\btau_1^{n-1}+\btau_2^{n-1}\|_0^2\right).\ 
\end{eqnarray}

{\em Second step of bounding the temporal derivative of the velocity difference.}
Adding terms in \eqref{eq;derit5} from $n=q+1,\ldots,M$ leads to 
\begin{eqnarray}\label{eq:derit6}
\lefteqn{\nu |\bE^M|_{G,1}^2+\frac{(15-16\eta_q)}{16}\Delta t\sum_{n=q+1}^M  \|\overline \partial_q\be_h^n\|_0^2
\le \nu \left|\bE^q\right|_{G,1}^2}\nonumber\\
&& +\frac{\eta_q}{2}\Delta t \|\overline \partial_q\be_h^q\|_0^2
+C \Delta t \sum_{n=q}^M \left(\|\nabla \be_h^n\|_0^2+\|\btau_1^n+\btau_2^n\|_0^2\right).
\end{eqnarray}
We will bound next the second term on the right-hand side of \eqref{eq:derit6}. To this end, we consider the error equation
\eqref{fully_error_gd} for $n=q$ and take $\bv_h=\overline\partial_q\be_h^q \in V_h^{\rm div}$
to get
\[
\|\overline\partial_q\be_h^q\|_0^2+\nu\(\nabla \be_h^q,\nabla\overline\partial_q\be_h^q\)+b\(\bs_h^q,\bs_h^q,\overline\partial_q\be_h^q\)
-b\(\bu_h^q,\bu_h^q,\overline\partial_q\be_h^q\)=\(\btau_1^q+\btau_2^q,\overline\partial_q\be_h^q\).
\]
Using the definition \eqref{eq_bdf_q} of the BDF methods, 
we can write the second term on the left-hand side above in the following way
\[
\nu\(\nabla \be_h^q,\nabla\overline\partial_q\be_h^q\)=\frac{\delta_0}{\Delta t}\nu \|\nabla \be_h^q\|_0^2
+\frac{1}{\Delta t}\nu\(\nabla \be_h^q,\sum_{i=1}^q \delta_i \nabla \be_h^{q-i}\),
\]
so that
\begin{eqnarray*}
\lefteqn{ \Delta t \|\overline\partial_q\be_h^q\|_0^2+{\delta_0}\nu \|\nabla \be_h^q\|_0^2
= -\nu\(\nabla \be_h^q,\sum_{i=1}^q \delta_i \nabla \be_h^{q-i}\)}\\
&&-\Delta t \left[b\(\bs_h^q,\bs_h^q,\overline\partial_q\be_h^q\)
-b\(\bu_h^q,\bu_h^q,\overline\partial_q\be_h^q\)\right]+\Delta t\(\btau_1^q+\btau_2^q,\overline\partial_q\be_h^q\).
\end{eqnarray*}
Utilizing standard techniques and already proved bounds, e.g., \eqref{aux2}, gives
\begin{eqnarray}\label{eq_ter1}
\Delta t \|\overline\partial_q\be_h^q\|_0^2+\nu\|\nabla \be_h^q\|_0^2\le C\left(\sum_{n=0}^{q-1}\nu\|\nabla \be_h^n\|_0^2+\Delta t \|\nabla \be_h^q\|_0^2+\Delta t\|\btau_1^q+\btau_2^q\|_0^2\right).
\end{eqnarray}
Inserting \eqref{eq_ter1} in \eqref{eq:derit6}, and since $15-16\eta_q\ge 2$, we arrive at
\begin{eqnarray*}
\nu |\bE^M|_{G,1}^2+\frac{\Delta t}{8}\sum_{n=q}^M  \|\overline \partial_q\be_h^n\|_0^2
&\le& \nu \left|\bE^q\right|_{G,1}^2+C\nu\sum_{n=1}^{q-1}\|\nabla \be_h^n\|_0^2\nonumber\\
&& +C\Delta t \sum_{n=q}^M \left(\|\nabla \be_h^n\|_0^2+\|\btau_1^n+\btau_2^n\|_0^2\right).
\end{eqnarray*}

Applying \eqref{equi_nor}, we conclude that there is a positive constant $C$ such that 
\begin{eqnarray}\label{eq:prindt}
\nu \|\nabla \be_h^M\|_0^2+\Delta t\sum_{n=q}^M  \|\overline \partial_q\be_h^n\|_0^2
\le C\left(\nu\sum_{n=1}^{q}\|\nabla \be_h^n\|_0^2+\sum_{n=q}^M \Delta t\left(\|\nabla \be_h^n\|_0^2+\|\btau_1^n+\btau_2^n\|_0^2\right)\right).
\end{eqnarray}
Again, it will be assumed that the initial conditions are sufficiently accurate, i.e., that 
\begin{equation}\label{con_ini_norm1}
\sum_{n=1}^{q-1}\nu\|\nabla \be_h^n\|_0^2\le C_{\mathrm{ic}} \(h^{2k}+(\Delta t)^{2q}\).
\end{equation}
The term $\nu \|\nabla \be_h^q\|_0^2$ was already bounded in \eqref{eq_ter1} by the initial conditions
and the truncation errors. 
To bound the first truncation error, we argue as in \eqref{cota_trun_1}--\eqref{cota_trun_3} and apply \eqref{er_pro_u} to show
\begin{equation}\label{tau1_nor0}
\sum_{n=q}^M\Delta t \|\btau_1^n\|_{0}^2\le C h^{2k}\Delta t\sum_{n=q}^M\(\|\bu_t^n\|_k^2+\|p_t^n\|_{k-1}^2\)
+C (\Delta t)^{2q}\int_0^T\|\partial_t^{q+1}\bu(s)\|_0^2 \ \di s.
\end{equation}
For the other truncation error, we use the approach presented in \eqref{asin3} and \eqref{asin4},
with the difference of getting on the right-hand side $\|\nabla(\bs_h^n-\bu^n)\|_0$
instead of $\|\bs_h^n-\bu^n\|_0$ and $\|\bv_h\|_0$ instead of $\|\nabla \bv\|_0$, i.e., of changing 
the role of these terms in the derivation of the bound, to prove
\begin{equation}\label{tau2_nor0}
\sum_{n=q}^M\Delta t \|\btau_2^n\|_{0}^2\le C h^{2k}\Delta t\sum_{n=q}^M\(\|\bu^n\|_{k+1}^2+\|p^n\|_{k}^2\).
\end{equation}
Applying \eqref{eq:cota_fin}, \eqref{con_ini_norm1}, \eqref{tau1_nor0}, and \eqref{tau2_nor0} 
at the right-hand side of \eqref{eq:prindt} 
gives
\begin{eqnarray}\label{eq:prindt2}
\lefteqn{\nu \|\nabla \be_h^M\|_0^2+\Delta t\sum_{n=q}^M  \|\overline \partial_q\be_h^n\|_0^2
\le C\(C_{\mathrm{ic}}+\frac{C_{\mathrm{e}}}{\nu}\) \left(h^{2k}+(\Delta t)^{2q}\right)}\nonumber\\
&&+C h^{2k}\Delta t\left(\sum_{n=q}^M \|\bu_t^n\|_k^2+\|\bu^n\|_{k+1}^2+\|p_t^n\|_{k-1}^2+\|p^n\|_{k}^2\right)+C (\Delta t)^{2q}\int_0^T\|\partial_t^{q+1}\bu(s)\|_0^2 \ \di s\nonumber\\
&\le& C_{\rm et}\left(h^{2k}+(\Delta t)^{2q}\right).
\end{eqnarray}

{\em Insertion of the individual bounds.} From \eqref{co_pre1} and the bounds \eqref{eq:press_02},
\eqref{eq:press_04}, and \eqref{eq:prindt2}, we finally obtain
\begin{equation*}
\Delta t \sum_{n=q}^M \|p_h^n-p_h^{s,n}\|_0^2\le C\(h^{2k}+(\Delta t )^{2q}\).
\end{equation*}

{\em Application of the triangle inequality.} Like at the end of the proof of Theorem~\ref{thm:velo},
the triangle inequality finishes the proof. 
\end{proof}

\begin{remark}\label{remark2} Concerning the time-step restrictions for getting the error bound for the pressure, apart from those assumed for the velocity, see Remark~\ref{remark1}, we need to assume
\eqref{cota_t_pre} which implies assumption $\Delta t \le C h^{3/(2q)}$, which in the case $d=2$ is slightly stronger than assumption
$\Delta t \le C h^{d/(2q)}$ in Remark~\ref{remark1}.
\end{remark}

\section{Error analysis of the method. Robust case}\label{sec:6}

In this section we consider the standard Galerkin method with grad-div stabilization. As shown in \cite{nos_grad_div,nos_bdf}
the use of grad-div stabilization allows to prove bounds with constants independent of inverse powers of the viscosity coefficient.
Those bounds were already shown in \cite{nos_bdf} for the BDF-2 method. In this section we prove robust error bounds for the non $A$-stable higher order methods BDF-$q$, $q\in\{3,4,5\}$,
thanks to Lemma~\ref{lema_bosco}. As already mentioned in the introduction, the
derivation of robust bounds can be considered as  a necessary condition for a good performance of a method in realistic problems. 

\begin{Theorem}[Robust error bound for the velocity]\label{thm:velob}
Assume that 
\begin{eqnarray*}
&&\bu \in L^\infty(0,T;H^2(\Omega)^d) \cap l^2(0,T;H^{k+1}(\Omega)^d), \
p \in l^2(0,T;H^{k}(\Omega)),\nonumber\\
&&\bu_t\in  H^{q+1}(0,T;H^1(\Omega)^d) \cap l^2(0,T;H^k(\Omega)^d).\
\nonumber
\end{eqnarray*}
Let the time step be sufficiently small, where the concrete restrictions are given in \eqref{CFL1}, \eqref{res_t}, 
\eqref{restri_delta}, \eqref{deltatul}, and \eqref{cond_htb} below in the proof. In addition, it will be assumed that the initial conditions are computed sufficiently accurately, concretely, 
that there exits a constant $C_{\mathrm{ic}}>0$ such that
\begin{equation}\label{inib}
\sum_{n=0}^{q-1} \|\bs_h^{m,n}-\bu_h^n\|_0^2\le C_{\mathrm{ic}} \(h^{2k}+(\Delta t)^{2q}\),
\end{equation}
where $\bs_h^{m,n}$ is the velocity component of the Stokes projection. Then the following 
estimate of the velocity error holds
\begin{equation}\label{eq:velo_estb}
\|\bu(T) - \bu_h^M\|_0^2 +  \Delta t \nu\sum_{n=q}^M\left\|\nabla \(\bu(t_n) - \bu_h^n\)\right\|_0^2 \le
C \(h^{2k}+(\Delta t)^{2q}\),
\end{equation}
where $k$ is the degree of the piecewise polynomials of the velocity space and the  constant
on the right-hand side of \eqref{eq:velo_estb} depends on some norms of the velocity and pressure but not explicitly on $\nu^{-1}$.  
\end{Theorem}

\begin{proof}
{\em Derivation of the error equation.}
We take
in \eqref{eq:bdfq_grad_div} a positive grad-div parameter $\mu>0$.
For the error analysis we will compare $\bu_h^n$ with the  velocity component of the modified Stokes projection 
$\bs_h^{m,n} := \bs_h^m(t_n)$, for which error bounds with constants independent of the inverse of $\nu^{-1}$ hold, see \eqref{eq:cotanewpro}--\eqref{nabla_sh_mod}.
Using \eqref{eq:gal_stokes}, \eqref{eq:stokes_rhs_g}, and $\nabla \cdot  \bu^{n} = 0$, one derives 
for $n\ge q$, $\bv_h\in V_h^{\rm div}$, and $\pi_h^n=\pi_h(t_n)$, see \eqref{eq:pressurel2},
\begin{eqnarray}\label{eq:fullydiscrete_shm}
\lefteqn{\(\overline \partial_q\bs_h^{m,n},\bv_h\)+\nu\(\nabla \bs_h^{m,n},\nabla \bv_h\)+b\(\bs_h^{m,n},\bs_h^{m,n},\bv_h\)
+\mu\(\nabla \cdot \bs_h^{m,n},\nabla \cdot \bv_h\)}\nonumber\\
&=&\(\boldsymbol f^{n},\bv_h\)-\(\nabla \cdot \bv_h,p^{n}-\pi_h^{n}\)
+\mu(\nabla \cdot (\bs_h^{m,n}-\bu^{n}),\nabla \cdot \bv_h)\\
&&+\(\overline \partial_q\bs_h^{m,n}-\bu_t^{n},\bv_h\)
+b\(\bs_h^{m,n},\bs_h^{m,n},\bv_h\)-b\(\bu^{n},\bu^{n},\bv_h\).
\nonumber
\end{eqnarray}
Denoting by 
$
\be_h^n=\bs_h^{m,n}-\bu_h^n \in V_h^{\rm div}
$
and subtracting \eqref{eq:bdfq_grad_div} from \eqref{eq:fullydiscrete_shm} we get
\begin{eqnarray}\label{fully_error_gd_b}
\lefteqn{\hspace*{-9em}\(\overline \partial_q\be_h^{n},\bv_h\)+\nu\(\nabla \be_h^{n},\nabla \bv_h\)
+b\(\bs_h^{m,n}, \bs_h^{m,n},\bv_h\)-b\(\bu_h^n,\bu_h^n,\bv_h)\)}\nonumber\\
+\mu\(\nabla \cdot \be_h^{n},\nabla \cdot \bv_h\) &=&\(\btau_1^{n}+\btau_2^{n},\bv_h\)+\(\btau_3^{n},\nabla \cdot \bv_h\),
\end{eqnarray}
with  $\btau_1^{n}$ and $\btau_2^{n}$ as in \eqref{lostrun}, with $\bs_h^n$ replaced by $\bs_h^{m,n}$, and 
\begin{equation*}
\btau_3^{n}=-\(p^{n}-\pi_h^{n}\)+\mu\nabla \cdot\(\bs_h^{m,n}-\bu^{n}\).
\end{equation*} 

{\em First step. Deriving a bound for $\sum_{n=q+1}^M\|D \be_h^n\|_0^2.$}
We first take in \eqref{fully_error_gd_b} $\bv_h=D\be_h^n$. Applying \eqref{cota_Aesta} from Lemma~\ref{lema_bosco} 
for $H= L^2(\Omega)^d$ yields 
\begin{eqnarray*}
 \lefteqn {s_q \|D \be_h^n\|_0^2+\|D\bE^n\|_{G^q}^2-\|D\bE^{n-1}\|_{G^q}^2+{\Delta t }\nu\(\nabla \be_h^n,\nabla D\be_h^n\)
 +{\Delta t }\mu\(\nabla \cdot\be_h^n,\nabla \cdot D\be_h^n\)}\\
 &\le & {\Delta t }\left[ b\(\bu_h^n,\bu_h^n,D\be_h^n\)-b\(\bs_h^{m,n}, \bs_h^{m,n},D\be_h^n\)\right)]+{\Delta t }\(\btau_1^{n}+\btau_2^{n},D\be_h^n\)+{\Delta t }\(\btau_3^{n},\nabla \cdot D\be_h^n\).
\end{eqnarray*}
Since for any function $\bv$ it holds
\[
(\bv^n,D\bv^n)=\frac{1}{2}\|\bv^n\|_0^2-\frac{1}{2}\|\bv^{n-1}\|_0^2+\frac{1}{2}\|\bv^n-\bv^{n-1}\|_0^2\ge
\frac{1}{2}\|\bv^n\|_0^2-\frac{1}{2}\|\bv^{n-1}\|_0^2,
\]
we obtain
\begin{eqnarray}
\lefteqn{s_q\|D \be_h^n\|_0^2+\|D\bE^n\|_{G^q}^2-\|D\bE^{n-1}\|_{G^q}^2+\frac{\Delta t }{2}\nu\|\nabla \be_h^n\|_0^2
 -\frac{\Delta t }{2}\nu\|\nabla \be_h^{n-1}\|_0^2}\nonumber\\
 &&+\frac{\Delta t }{2}\mu\|\nabla \cdot\be_h^n\|_0^2
 -\frac{\Delta t }{2}\mu\|\nabla \cdot\be_h^{n-1}\|_0^2\label{er_ro}\\
 &\le& {\Delta t }\left[ b\(\bu_h^n,\bu_h^n,D\be_h^n\)-b\(\bs_h^{m,n}, \bs_h^{m,n},D\be_h^n\)\right]+{\Delta t }\(\btau_1^{n}+\btau_2^{n},D\be_h^n\)+{\Delta t }\(\btau_3^{n},\nabla \cdot D\be_h^n\).\nonumber
\end{eqnarray}

To bound the nonlinear term, arguing as in \eqref{nonli3med} and applying \eqref{cota_non_1}, \eqref{cota_non_3},
and \eqref{cota_non_2}, we find that 
\begin{eqnarray}\label{non_ca_1}
\lefteqn{\left|b\(\bu_h^n,\bu_h^n,D\be_h^n\)-b\(\bs_h^{m,n}, \bs_h^{m,n},D\be_h^n)\)\right|} \nonumber \\
&\le& \left|b\(\be_h^n,\be_h^n,D\be_h^n\)\right|+\left|b\(\bs_h^{m,n},\be_h^n,D\be_h^n\)\right| +\left|b\(\be_h^n,\bs_h^{m,n},D\be_h^n\)\right|\\
&\le& \left(2\|\be_h^n\|_\infty+
{2}\|\bs_h^{m,n}\|_\infty+C\(\|\nabla \cdot \bs_h^{m,n}\|_{{L^{2d/(d-1)}}}+\|\nabla \bs_h^{m,n}\|_{L^{2d/(d-1)}}\)\right)\|\nabla \be_h^n\|_0
\|D\be_h^n\|_0.\nonumber
\end{eqnarray}
It will be proved in the final part of this proof that 
for a fixed constant $c_\infty$ holds
\begin{equation}\label{cotaero}
\|\be_h^n\|_\infty\le c_\infty,\quad n=0,\ldots,M,
\end{equation}
whenever $\Delta t$ and $h$ are sufficiently small. 
Inserting \eqref{cotaero} and \eqref{cotas_s_h_mod} in \eqref{non_ca_1}, we reach
\[
\left|b\(\bu_h^n,\bu_h^n,D\be_h^n)-b(\bs_h^{m,n}, \bs_h^{m,n},D\be_h^n)\)\right|\le c_{1,\bu}\|\nabla \be_h^n\|_0\|D\be_h^n\|_0,
\]
where
\[
c_{1,\bu}=2c_\infty+{4} \esssup{0\le t\le T}\|\bu(t)\|_\infty +C\ \esssup{0\le t\le T}\|\bu(t)\|_2.
\]
As a consequence, applying the inverse inequality \eqref{inv} and Young's inequality gives
\begin{eqnarray*}
{\Delta t }\left| b\(\bu_h^n,\bu_h^n,D\be_h^n)-b(\bs_h^{m,n}, \bs_h^{m,n},D\be_h^n\)\right| 
&\le & \Delta tc_{1,\bu}{c_{\mathrm{inv}}}h^{-1}\|\be_h^n\|_0\|D\be_h^n\|_0\\
&\le& \frac{(\Delta t)^2}{s_q}c_{1,\bu}^2{c_{\mathrm{inv}}^2}h^{-2}\|\be_h^n\|_0^2+\frac{s_q}{4}\|D\be_h^n\|_0^2. \nonumber
\end{eqnarray*}

Assuming
\begin{equation}\label{CFL1}
\frac{\Delta t}{s_q} c_{1,\bu}{c_{\mathrm{inv}}^2}h^{-2}<1, 
\end{equation}
we finally obtain
\begin{equation}\label{eq:b1}
{\Delta t }\left| b\(\bu_h^n,\bu_h^n,D\be_h^n\)-b\(\bs_h^{m,n}, \bs_h^{m,n},D\be_h^n\)\right| \le \Delta t c_{1,\bu}\|\be_h^n\|_0^2+\frac{s_q}{4}\|D\be_h^n\|_0^2.
\end{equation}

For the first truncation errors, we estimate
\[
{\Delta t }\left|\(\btau_1^{n}+\btau_2^{n},D\be_h^n\)\right|\le \Delta t\|\btau_1^{n}+\btau_2^{n}\|_0^2+\frac{\Delta t}{4}\|D\be_h^n\|_0^2,
\]
and assuming
\begin{equation}\label{res_t}
{\Delta t}<s_q,
\end{equation}
we reach
\begin{eqnarray}\label{eq:t1}
{\Delta t }\left|\(\btau_1^{n}+\btau_2^{n},D\be_h^n\)\right|\le \Delta t\|\btau_1^{n}+\btau_2^{n}\|_0^2+\frac{s_q}{4}\|D\be_h^n\|_0^2.
\end{eqnarray}
Using the inverse inequality \eqref{inv} to bound the last truncation error yields 
\[
{\Delta t }\(\btau_3^{n},\nabla \cdot D\be_h^n\)\le \Delta t \|\btau_3^{n}\|_0 c_{\mathrm{inv}} h^{-1}\|D\be_h^n\|_0
\le \frac{(\Delta t)^2}{s_q}{c_{\mathrm{inv}}^2}h^{-2} \|\btau_3^{n}\|_0 ^2+\frac{s_q}{4}\|D\be_h^n\|_0^2.
\]
And, under condition \eqref{CFL1}, assuming $c_{1,\bu}\ge 1$, we conclude
\begin{eqnarray}\label{eq:t2}
{\Delta t }\(\btau_3^{n},\nabla \cdot D\be_h^n\)
\le {\Delta t}\|\btau_3^{n}\|_0 ^2+\frac{s_q}{4}\|D\be_h^n\|_0^2.
\end{eqnarray}
Inserting \eqref{eq:b1}, \eqref{eq:t1}, and \eqref{eq:t2} in \eqref{er_ro}, we obtain
\begin{align*}
 &\frac{s_q}{4}\|D \be_h^n\|_0^2+\|D\bE^n\|_{G^q}^2-\|D\bE^{n-1}\|_{G^q}^2+\frac{\Delta t }{2}\nu\|\nabla \be_h^n\|^2
 -\frac{\Delta t }{2}\nu\|\nabla \be_h^{n-1}\|^2\nonumber\\
 &\quad +\frac{\Delta t }{2}\mu\|\nabla \cdot\be_h^n\|^2
 -\frac{\Delta t }{2}\mu\|\nabla \cdot\be_h^{n-1}\|^2 \le\Delta t c_{1,\bu}\|\be_h^n\|_0^2+\Delta t\|\btau_1^{n}+\btau_2^{n}\|_0^2
 +{\Delta t}\|\btau_3^{n}\|_0 ^2.\nonumber
\end{align*}
Adding from $n=q+1$ to $M$ and neglecting nonnegative terms on the left-hand side gives
\begin{eqnarray*}
\sum_{n=q+1}^M\|D \be_h^n\|_0^2 &\le& \frac{2}{s_q}\left(2\|D\bE^{q}\|_{G^q}^2+{\Delta t }\nu\|\nabla \be_h^q\|_0^2+{\Delta t }\mu\|\nabla \cdot\be_h^q\|_0^2\right)\\
&&+\frac{4\Delta t}{s_q}c_{1,\bu}\sum_{n=q+1}^M \|\be_h^n\|_0^2+\frac{4\Delta t}{s_q}\sum_{n=q+1}^M
\left(\|\btau_1^{n}+\btau_2^{n}\|_0^2+\|\btau_3^{n}\|_0 ^2\right).
\end{eqnarray*}
Denoting by $\lambda_{\rm max}^{(q)}$ the largest eigenvalue of $G^q$, we obtain
with \eqref{norma_G} and \eqref{equi_nor} 
\[
\|D\bE^{q}\|_{G^q}^2\le \lambda_{\rm max}^{(q)}\sum_{n=1}^q \|D \be_h^{n}\|_0^2\le C\sum_{n=0}^q \|\be_h^n\|_0^2,
\]
so that 
\begin{eqnarray}\label{eq:dif_def}
\sum_{n=q+1}^M\|D \be_h^n\|_0^2&\le& C\sum_{n=0}^{q-1}\|\be_h^n\|_0^2+C\left(\|\be_h^q\|_0^2+\Delta t\nu\|\nabla \be_h^q\|_0^2+\mu\Delta t\|\nabla \cdot\be_h^q\|_0^2\right)\nonumber\\
&& +\frac{4\Delta t}{s_q}c_{1,\bu}\sum_{n=q+1}^M \|\be_h^n\|_0^2+\frac{4\Delta t}{s_q}\sum_{n=q+1}^M
\left(\|\btau_1^{n}+\btau_2^{n}\|_0^2+\|\btau_3^{n}\|_0 ^2\right).
\end{eqnarray}

{\em Second step. Estimation of the error.} The following error analysis follows principally the one from 
the proof of Theorem~\ref{thm:velo}, with some differences to exploit the appearance of the grad-div stabilization. 
We first observe that
\begin{equation}\label{eq:rho_q}
\be_h^n-\eta_q\be_h^{n-1}=(1-\eta_q)\be_h^n+\eta_q D \be_h^n=\rho_q\be_h^n+\eta_q D \be_h^n \quad
\mbox{with}\quad 0<\rho_q=1-\eta_q<1.
\end{equation}
Taking in \eqref{fully_error_gd_b} $\bv_h=\be_h^n-\eta_q\be_h^{n-1}$ yields
\begin{eqnarray*}
\lefteqn{
\(\overline \partial_q\be_h^{n},\be_h^n-\eta_q\be_h^{n-1}\)+\nu\rho_q\|\nabla \be_h^{n}\|_0^2+\mu\rho_q\|\nabla \cdot \be_h^{n}\|_0^2+\nu\frac{\eta_q}{2}\|\nabla \be_h^{n}\|_0^2+\mu\frac{\eta_q}{2}\|\nabla \cdot\be_h^{n}\|_0^2} \\
&\le& \nu\frac{\eta_q}{2}\|\nabla \be_h^{n-1}\|_0^2+\mu\frac{\eta_q}{2}\|\nabla \cdot\be_h^{n-1}\|_0^2 +b\(\bu_h^n,\bu_h^n,\be_h^n-\eta_q\be_h^{n-1}\)\\
&&-b\(\bs_h^{m,n}, \bs_h^{m,n},\be_h^n-\eta_q\be_h^{n-1}\)+\(\btau_1^{n}+\btau_2^{n},\be_h^n-\eta_q\be_h^{n-1}\)+\(\btau_3^{n},\nabla \cdot \(\be_h^n-\eta_q\be_h^{n-1}\)\).
\end{eqnarray*}
Applying the definition of $\overline \partial_q$ from 
\eqref{eq_bdf_q} and Lemma~\ref{lema_3_6} in combination with Lemma~\ref{lem:mult_tech}
with $H=L^2(\Omega)$ gives
\begin{eqnarray}\label{eq:prinro}
\lefteqn{\hspace*{-4.5em}\left|\bE_h^n\right|_G^2-\left|\bE_h^{n-1}\right|_G^2+\nu\rho_q\Delta t \|\nabla \be_h^{n}\|_0^2+\mu\rho_q\Delta t \|\nabla \cdot \be_h^{n}\|_0^2+\nu\Delta t \frac{\eta_q}{2}\|\nabla \be_h^{n}\|_0^2+\mu\Delta t \frac{\eta_q}{2}\|\nabla \cdot\be_h^{n}\|_0^2}\nonumber\\
&\le& \nu\Delta t \frac{\eta_q}{2}\|\nabla \be_h^{n-1}\|_0^2+\mu\Delta t \frac{\eta_q}{2}\|\nabla \cdot\be_h^{n-1}\|_0^2\nonumber\\
&&+\Delta t \left[b\(\bu_h^n,\bu_h^n,\be_h^n-\eta_q\be_h^{n-1}\)-b\(\bs_h^{m,n}, \bs_h^{m,n},\be_h^n-\eta_q\be_h^{n-1}\)\right]\nonumber\\
&&+\Delta t \(\btau_1^{n}+\btau_2^{n},\be_h^n-\eta_q\be_h^{n-1}\) +\Delta t \(\btau_3^{n},\nabla \cdot(\be_h^n-\eta_q\be_h^{n-1})\).
\end{eqnarray}

To bound the nonlinear term, we first apply \eqref{eq:rho_q} to get
\begin{eqnarray}\label{nonli1b}
\lefteqn{b\(\bu_h^n,\bu_h^n,\be_h^n-\eta_q\be_h^{n-1}\)-b\(\bs_h^{m,n}, \bs_h^{m,n},\be_h^n-\eta_q\be_h^{n-1}\)}\\
&=&\rho_q\(b\(\bu_h^n,\bu_h^n,\be_h^n)-b(\bs_h^{m,n}, \bs_h^{m,n},\be_h^n\)\)
+\eta_q\(b\(\bu_h^n,\bu_h^n,D\be_h^n)-b(\bs_h^{m,n}, \bs_h^{m,n},D\be_h^n\)\).\nonumber
\end{eqnarray}
For the first term on the right-hand side, using the first step of \eqref{nonli2}, the  skew-symmetric property of the bilinear term, and applying standard estimates, we obtain  
\begin{eqnarray*}
\left|b\(\bu_h^n,\bu_h^n,\be_h^n)-b(\bs_h^{m,n}, \bs_h^{m,n},\be_h^n\)\right|&\le&\|\nabla \bs_h^{m,n}\|_\infty\|\be_h^n\|_0^2+
\| \bs_h^{m,n}\|_\infty\|\nabla \cdot\be_h^n\|_0\|\be_h^n\|_0\\
&\le&\left(\|\nabla \bs_h^{m,n}\|_\infty+\frac{8}{\mu\rho_q}\|\bs_h^{m,n}\|_\infty^2\right)\|\be_h^n\|_0^2+\frac{\mu\rho_q}{32}\|\nabla \cdot\be_h^n\|_0^2.\nonumber
\end{eqnarray*}
Notice that here the grad-div stabilization term is utilized to avoid a dependency of the 
error bound on inverse powers of the viscosity coefficient. 
Applying \eqref{cotas_s_h_mod} and \eqref{nabla_sh_mod} leads to 
\begin{equation}\label{nonli2b}
\left|b\(\bu_h^n,\bu_h^n,\be_h^n)-b(\bs_h^{m,n}, \bs_h^{m,n},\be_h^n\)\right|\le c_{2,\bu}\|\be_h^n\|_0^2+\frac{\mu\rho_q}{32}\|\nabla \cdot\be_h^n\|_0^2,
\end{equation}
where
$$
c_{2,\bu}=\(C_\infty+1\)\esssup{0\le t\le T}\|\nabla \bu(t)\|_\infty+\frac{32}{\mu\rho_q}\esssup{0\le t\le T}\|\bu(t)\|_\infty^2.
$$
The second term in \eqref{nonli1b} was already bounded in \eqref{eq:b1}, so that
\begin{equation}\label{nonli3b}
\left| b\(\bu_h^n,\bu_h^n,D\be_h^n\)-b\(\bs_h^{m,n}, \bs_h^{m,n},D\be_h^n\)\right| \le  c_{1,\bu}\|\be_h^n\|_0^2+\frac{s_q}{4\Delta t }\|D\be_h^n\|_0^2.
\end{equation}
Summarizing \eqref{nonli1b}, \eqref{nonli2b}, and \eqref{nonli3b}, and taking into account that $\rho_q,\eta_q<1$,
we get
\begin{eqnarray}\label{non_rodef}
\lefteqn{\Delta t \left[b\(\bu_h^n,\bu_h^n,\be_h^n-\eta_q\be_h^{n-1}\)-b\(\bs_h^{m,n}, \bs_h^{m,n},\be_h^n-\eta_q\be_h^{n-1}\)\right]}\nonumber\\
&\le& \(c_{1,\bu}+c_{2,\bu}\)\Delta t\|\be_h^n\|_0^2+\frac{\mu\rho_q}{32}\Delta t \|\nabla \cdot\be_h^n\|_0^2+\frac{s_q}{4}{\|D \be_h^n\|_0^2}.
\end{eqnarray}

For estimating the truncation errors in \eqref{eq:prinro}, we use \eqref{eq:rho_q}, the Cauchy--Schwarz inequality, 
Young's inequality, the inverse inequality \eqref{inv},
$\rho_q+\eta_q=1$, condition \eqref{CFL1}, and assume $c_{1,\bu}\ge1$,
\begin{eqnarray}\label{tru_ro}
\lefteqn{\Delta t\left| \(\btau_1^{n}+\btau_2^{n},\be_h^n-\eta_q\be_h^{n-1}\) +\Delta t \(\btau_3^{n},\nabla \cdot\(\be_h^n-\eta_q\be_h^{n-1}\)\)\right|}
\nonumber\\
& = &\Delta t  \left|\rho_q\(\btau_1^{n}+\btau_2^{n},\be_h^n\)+\rho_q\(\btau_3^{n},\nabla \cdot \be_h^n \)+\eta_q\(\btau_1^{n}+\btau_2^{n},D\be_h^n\)+\eta_q\(\btau_3^{n},\nabla \cdot D\be_h^n\)\right|\nonumber\\
&\le & \Delta t \frac{\rho_q}{2}\|\btau_1^{n}+\btau_2^{n}\|_0^2+\Delta t \frac{\rho_q}{2}\|\be_h^n\|_0^2
+\Delta t \frac{\eta_q}{2}\|\btau_1^{n}+\btau_2^{n}\|_0^2+\Delta t \frac{\eta_q}{2}\|D\be_h^n\|_0^2\nonumber\\
&&+\Delta t \frac{8\rho_q}{\mu}\|\btau_3^n\|_0^2+\Delta t \frac{\rho_q\mu}{32}\|\nabla \cdot\be_h^n\|_0^2
+\Delta t \eta_q \|\btau_3^n\|_0 c_{\rm inv}h^{-1}\|D \be_h^n\|_0\nonumber\\
&\le& \frac{\Delta t}{2} \|\btau_1^{n}+\btau_2^{n}\|_0^2+ \frac{\Delta t}{2}\|\be_h^n\|_0^2+\Delta t \frac{\mu\rho_q}{32}\|\nabla \cdot\be_h^n\|_0^2+\frac{1}{2}\(1+{\Delta t}\)\|D\be_h^n\|_0^2\nonumber\\
&&+\Delta t \(\frac{8\rho_q}{\mu}+\frac{s_q}{2}\)\|\btau_3^n\|_0^2.
\end{eqnarray}

Inserting \eqref{non_rodef} and \eqref{tru_ro} in \eqref{eq:prinro} and denoting by
\[
c^{\rm div}_q=\frac{15}{16}\rho_q,\ c_{3,\bu}=c_{1,\bu}+c_{2,\bu}+\frac{1}{2},\ c^{D}_q=\frac{s_q}{4}+ \frac{1}{2}(1+\Delta t),
\]
we reach
\begin{eqnarray*}
\lefteqn{\hspace*{-8em}\left|\bE_h^n\right|_G^2-\left|\bE_h^{n-1}\right|_G^2+\nu\rho_q\Delta t \|\nabla \be_h^{n}\|_0^2+\mu c_q^{\rm div}\Delta t \|\nabla \cdot \be_h^{n}\|_0^2
+\nu\Delta t \frac{\eta_q}{2}\|\nabla \be_h^{n}\|_0^2}\nonumber\\
+\mu\Delta t \frac{\eta_q}{2}\|\nabla \cdot\be_h^{n}\|_0^2
&\le& \nu\Delta t \frac{\eta_q}{2}\|\nabla \be_h^{n-1}\|_0^2+\mu\Delta t \frac{\eta_q}{2}\|\nabla \cdot\be_h^{n-1}\|_0^2+\Delta t c_{3,\bu}\|\be_h^n\|_0^2\nonumber\\
&& + c^D_q\|D \be_h^n\|_0^2+ \frac{\Delta t }{2}\|\btau_1^{n}+\btau_2^{n}\|_0^2+\Delta t  \(\frac{8\rho_q}{\mu}+\frac{s_q}{2}\)\|\btau_3^n\|_0^2.
\end{eqnarray*}
Adding terms from $n=q+1,\ldots,M$ and applying \eqref{equi_nor} yields
\begin{eqnarray*}
\lefteqn{\lambda_{\rm min}\|\be_h^M\|_0^2+\nu\rho_q\Delta t \sum_{n=q+1}^M\|\nabla \be_h^n\|_0^2+\mu c_q^{\rm div}\Delta t\sum_{n=q+1}^M\|\nabla \cdot\be_h^n\|_0^2 \le \lambda_{\rm max}\sum_{n=1}^{q-1}\|\be_h^n\|_0^2}\\
&& +\lambda_{\rm max}\|\be_h^q\|_0^2+\nu\Delta t \frac{\eta_q}{2}\|\nabla \be_h^{q}\|_0^2+\mu\Delta t \frac{\eta_q}{2}\|\nabla \cdot\be_h^{q}\|_0^2
+\Delta t c_{3,\bu}\sum_{n=q+1}^M\|\be_h^n\|_0^2\nonumber\\
&&+c^D_q\sum_{n=q+1}^M\|D \be_h^n\|_0^2+C \Delta t	\sum_{n=q+1}^M  \|\btau_1^{n}+\btau_2^{n}\|_0^2+C\Delta t	\sum_{n=q+1}^M \|\btau_3^n\|_0^2.
\end{eqnarray*}

{\em Third step. Joining estimations of two previous steps.}
Inserting \eqref{eq:dif_def}, we reach
\begin{eqnarray}\label{estasi}
\lefteqn{\lambda_{\rm min}\|\be_h^M\|_0^2+\nu\rho_q\Delta t \sum_{n=q+1}^M\|\nabla \be_h^n\|_0^2+\mu c_q^{\rm div}\Delta t\sum_{n=q+1}^M\|\nabla \cdot\be_h^n\|_0^2\le C\sum_{n=1}^{q-1}\|\be_h^n\|_0^2}\nonumber\\
&& +C\left(\|\be_h^q\|_0^2+\Delta t\nu\|\nabla \be_h^q\|_0^2+\mu\Delta t\|\nabla \cdot\be_h^q\|_0^2\right)
+\Delta t c_{4,\bu}\sum_{n=q+1}^M\|\be_h^n\|_0^2\nonumber\\
&&+C \Delta t	\sum_{n=q+1}^M  \|\btau_1^{n}+\btau_2^{n}\|_0^2 +C\Delta t	\sum_{n=q+1}^M \|\btau_3^n\|_0^2,
\end{eqnarray}
where
$c_{4,\bu}=c_{3,\bu}+\nicefrac{(4c_q^D c_{1,\bu})}{s_q}$.
To estimate the first three terms on the second line of \eqref{estasi}, we consider the error
equation \eqref{fully_error_gd_b} for $n=q$ and $\bv_h=\be_h^q$. Arguing similarly to \eqref{nq} yields
\begin{eqnarray}\label{nqb}
\lefteqn{\delta_0 \|\be_h^q\|_0^2+\Delta t \nu\|\nabla \be_h^{q}\|_0^2+\Delta t \mu \|\nabla \cdot \be_h^q\|_0^2\le
\Delta t\left|b\(\bs_h^{m,q}, \bs_h^{m,q},\be_h^q\)-b\(\bu_h^q,\bu_h^q,\be_h^q\)\right|}\nonumber\\
&&+\Delta t\|\btau_1^{q}+\btau_2^{q}\|_{0}\|\be_h^q\|_0+\Delta t \|\btau_3^q\|_0\|\nabla \cdot \be_h^q\|_0+\left\|\sum_{i=1}^q\delta_i\be_h^{q-i}\right\|_0\|\be_h^q\|_0.
\end{eqnarray}
Assuming 
\begin{equation}\label{restri_delta}
\Delta t \le \frac{\delta_0}{4c_{2,\bu}}
\end{equation}
gives, in combination with \eqref{nonli2b},
\begin{equation}\label{nqb2}
\Delta t\left|b\(\bs_h^{m,q}, \bs_h^{m,q},\be_h^q\)-b\(\bu_h^q,\bu_h^q,\be_h^q\)\right| 
\le \frac{\delta_0}{4}\|\be_h^q\|_0^2+\Delta t \frac{\mu}{32}\|\nabla \cdot\be_h^n\|_0^2.
\end{equation}
Inserting \eqref{nqb2} in \eqref{nqb}, applying Young's inequality and standard arguments, it is straightforward to conclude that there exists a constant $C_{\rm f}$ such that
\begin{equation}\label{erroresqf}
\|\be_h^q\|_0^2+\Delta t\nu\|\nabla \be_h^q\|_0^2+\mu\Delta t\|\nabla \cdot\be_h^q\|_0^2\le C_{\rm f}\sum_{n=0}^{q-1}\|\be_h^n\|_0^2
+C_{\rm f}\Delta t \|\btau_1^{q}+\btau_2^{q}\|_{0}^2+C_{\rm f}\Delta t \|\btau_3^q\|_0^2.
\end{equation}
This bound is inserted in  \eqref{estasi}, leading to 
\begin{eqnarray*}
\lefteqn{\|\be_h^M\|_0^2+\nu\Delta t \sum_{n=q+1}^M\|\nabla \be_h^n\|_0^2+\mu\Delta t\sum_{n=q+1}^M\|\nabla \cdot\be_h^n\|_0^2}\\
&\le& C\sum_{n=1}^{q-1}\|\be_h^n\|_0^2
+ C\Delta t c_{4,\bu}\sum_{n=q+1}^M\|\be_h^n\|_0^2+C \Delta t	\sum_{n=q}^M  \|\btau_1^{n}+\btau_2^{n}\|_0^2+C\Delta t	\sum_{n=q}^M \|\btau_3^n\|_0^2. \nonumber
\end{eqnarray*}
Applying now (Gronwall's) Lemma~\ref{gronwall} with $k=\Delta t$ and $\gamma_j=C c_{4,\bu}$, assuming
\begin{equation}\label{deltatul}
\Delta t \le \frac{1}{2Cc_{4,\bu}},
\end{equation}
so that $\Delta t \gamma_j\le 1/2$, and taking into account \eqref{erroresqf}, we reach
\begin{eqnarray}\label{final1}
\lefteqn{\|\be_h^M\|_0^2+\nu\Delta t \sum_{n=q}^M\|\nabla \be_h^n\|_0^2+\mu\Delta t\sum_{n=q}^M\|\nabla \cdot\be_h^n\|_0^2} \nonumber\\
&\le&  C e^{2TCc_{4,\bu}}\left(\sum_{n=0}^{q-1}\|\be_h^n\|_0^2+{\Delta t}\sum_{n=q}^M  \|\btau_1^{n}+\btau_2^{n}\|_0^2+ \Delta t	\sum_{n=q}^M \|\btau_3^n\|_0^2\right).
\end{eqnarray}

{\em Estimation of the terms on the right-hand side of \eqref{final1}.} For the first term on the right-hand side of \eqref{final1} we apply \eqref{inib}. To bound the first truncation error, we argue as in
\eqref{cota_trun_1}--\eqref{er_pro_u}, with bounding $\|\cdot\|_0$ instead of $\|\cdot\|_{-1}$, and applying \eqref{eq:cotanewpro}
to get
\begin{equation}\label{tau1_nor0b}
\sum_{n=q}^M\Delta t \|\btau_1^n\|_{0}^2\le C h^{2k}\Delta t\sum_{n=q}^M\|\bu_t^n\|_k^2
+ C(\Delta t)^{2q}\int_0^T\|\partial_t^{q+1}\bu(s)\|_0^2 \ \di s\le C_{\btau_1}\(h^{2k}+(\Delta t)^{2q}\).
\end{equation}
For the second truncation error we follow \eqref{tau2_nor0} with applying \eqref{eq:cotanewpro} to obtain
\begin{equation}\label{tau2_nor0b}
\sum_{n=q}^M\Delta t \|\btau_2^n\|_{0}^2\le C h^{2k}\Delta t\sum_{n=q}^M\|\bu^n\|_{k+1}^2\le C_{\btau_2}h^{2k}.
\end{equation}
Finally, to bound the third truncation error, we apply \eqref{eq:pressurel2} and \eqref{eq:cotanewpro}, which gives
\begin{equation}\label{tau3_nor0b}
\sum_{n=q}^M\Delta t \|\btau_3^n\|_{0}^2\le C h^{2k}\Delta t\sum_{n=q}^M\(\|p^n\|_{k}^2+\|\bu^n\|_{k+1}^2\)
\le C_{\btau_3}h^{2k}.
\end{equation}
Inserting \eqref{tau1_nor0b}, \eqref{tau2_nor0b}, and \eqref{tau3_nor0b} in \eqref{final1}, we conclude
\begin{equation}\label{final2}
\|\be_h^M\|_0^2+\nu\Delta t \sum_{n=q}^M\|\nabla \be_h^n\|_0^2+\mu\Delta t\sum_{n=q}^M\|\nabla \cdot\be_h^n\|_0^2\le  C_e \(h^{2k}+(\Delta t)^{2q}\),
\end{equation}
with $C_e$ depends on $T$, norms of $\bu$, $\bu_t$, and $p$ but not explicitly on $\nu^{-1}$.
With \eqref{final2}, the triangle inequality, and \eqref{eq:cotanewpro}, the estimate \eqref{eq:velo_estb}
is proved.

{\em Proof of \eqref{cotaero}.}
The derivation of \eqref{final2} used \eqref{cotaero}. We need to prove that for $j=0,\ldots,M$ 
the bound
\begin{equation}\label{cotaero2}
\|\be_h^j\|_\infty\le c_\infty,
\end{equation}
holds.
As we did with \eqref{apriori}, we will prove \eqref{cotaero2} by induction whenever the time step 
and the mesh size are small enough. Let us fix a constant $c_\infty$. Assume that $\Delta t$ and 
$h$ are sufficiently small so that 
\begin{equation}\label{cond_htb}
\begin{array}{rclrcl}
c_{\rm inv} C_{\mathrm{ic}}^{1/2} h^{k-d/2}&\le& \displaystyle \frac{c_\infty}{2}, &  K_1^{1/2}h^{k-d/2}&\le& \displaystyle \frac{c_\infty}{\sqrt{2}}\\[1em]
c_{\rm inv} C_{\mathrm{ic}}^{1/2} h^{-d/2}(\Delta t)^q&\le& \displaystyle \frac{c_\infty}{2},& K_2^{1/2}h^{-d/2}(\Delta t)^q&\le& \displaystyle \frac{c_\infty}{\sqrt{2}},
\end{array}
\end{equation}
where $$K_1=c_{\rm inv}^2 \(C_{\rm f}\max\{C_{\mathrm{ic}},C_e\}q+C_{\rm f}(2C_{\btau_1}+{2}C_{\btau_2}+C_{\btau_3})\right),\ K_2={C_{\rm f}\max\{C_{\mathrm{ic}},C_e\} q+}c_{\rm inv}^2 2C_fC_{\btau_1}.$$
Consider first the initial values. 
For $j=0,\ldots,q-1$, applying the inverse inequality \eqref{inv}, \eqref{inib}, and \eqref{cond_htb}
gives
\[
\|\be_h^j\|_\infty\le c_{\rm inv} h^{-d/2}\|\be_h^j\|_0\le c_{\rm inv} C_{\mathrm{ic}}^{1/2}h^{-d/2}\(h^{k}+(\Delta t)^q\)\le c_\infty,
\]
so that condition \eqref{cotaero2} is satisfied for $n=0,\ldots,q-1$. Assume now \eqref{cotaero2} holds for $j\le n-1$. Then, we can apply the arguments that gave \eqref{final2} for $j\le n-1$ so that 
$$
\|\be_h^j\|_0^2\le C_e \(h^{2k}+(\Delta t)^{2q}\),\quad  j\le n-1.
$$
Combining this inequality and \eqref{inib} with \eqref{erroresqf} for $q=n$ yields
\[
\|\be_h^n\|_0^2\le C_{\rm f}\max\left\{C_{\mathrm{ic}},C_e\right\}q\(h^{2k}+(\Delta t)^{2q}\)
+C_{\rm f}\Delta t \|\btau_1^{n}+\btau_2^{n}\|_{0}^2+C_{\rm f}\Delta t \|\btau_3^n\|_0^2.
\]
Applying the triangle inequality, \eqref{tau1_nor0b}, \eqref{tau2_nor0b}, and \eqref{tau3_nor0b}  leads to 
\[
\|\be_h^n\|_0^2\le C_{\rm f}\max\{C_{\mathrm{ic}},C_e\} q \(h^{2k}+(\Delta t)^{2q}\)
+ C_{\rm f}\(2C_{\btau_1}\(h^{2k}+(\Delta t)^{2q}\)+2C_{\btau_2}h^{2k} + C_{\btau_3}h^{2k}\).
\]
With inverse inequality and \eqref{cond_htb}, we obtain
\begin{eqnarray*}
\|\be_h^n\|_\infty^2&\le& c_{\rm inv}^2 h^{-d}\(C_{\rm f}\max\{C_{\mathrm{ic}},C_e\} q+C_{\rm f}(2C_{\btau_1}+{2}C_{\btau_2}+C_{\btau_3})\right)h^{2k}\\
&& +c_{\rm inv}^2 h^{-d}\({C_{\rm f}\max\{C_{\mathrm{ic}},C_e\} q} +  2C_{\rm f}C_{\btau_1}\)(\Delta t )^{2q}\nonumber\\
&\le& K_1 h^{-d} h^{2k}+K_2 h^{-d} (\Delta t)^{2q}\le c_{\infty},
\end{eqnarray*}
which concludes the proof.
\end{proof}

\begin{remark}\label{re:robu}
Comparing the non-robust and robust cases we can observe that the velocity error in $L^2(\Omega)$ behaves as $h^{k+1}$ in the first case, see \eqref{eq:velo_est},
and as $h^k$ in the second one, compare \eqref{eq:velo_estb}. The rate of convergence $k$ is the best that can be obtained with constants independent of $\nu^{-1}$ for
the Galerkin method with grad-div stabilization, see \cite{review_nos}. 

Concerning the time step restrictions, \eqref{CFL1}, \eqref{res_t}, 
\eqref{restri_delta}, \eqref{deltatul}, and \eqref{cond_htb}, we can summarize them as follows. We require a time step being sufficiently small and that the quantity $h^{-d/2}(\Delta t)^q$ remains bounded, i.e., 
\begin{equation}\label{restri1}
h^{-d/2}(\Delta t)^q\le C,
\end{equation}
as in the non-robust case, see Remark~\ref{remark1}. We also require the stronger condition \eqref{CFL1} that implies
\begin{equation}\label{restri2}
\Delta t h^{-2}\le C.
\end{equation}
This CFL-type condition arises when one proves bounds with constants independent of $\nu^{-1}$ and was already
found in \cite[Remark 3.9]{nos_bdf} for the IMEX-BDF-2 method. 

{An alternative way to derive
a robust bound would be to replace \eqref{CFL1} with the condition  $\Delta t\le C \nu$ for sufficiently 
small $C$. Then, instead of $\frac{4\Delta t}{s_q}c_{1,\bu}\sum_{n=q+1}^M \|\be_h^n\|_0^2$ we get  in 
\eqref{eq:dif_def} the term $c \Delta t \sum_{n=q+1}^M \nu\|\nabla\be_h^n\|_0^2$ with a sufficiently small constant $c$, which can be finally absorbed in the left-hand side and all other steps of the proof 
work as above.} 
\end{remark}

\begin{Theorem}\label{thpre2} {\bf (Error bound for the pressure)} Let the assumptions of Theorem~\ref{thm:velob} be satisfied and assume also the time step satisfies \eqref{cota_t_pre}, then 
\begin{equation}\label{eq:pres_bound_2}
\Delta t \sum_{n=q}^M \|p(t_n)-p_h^n\|_0^2\le C \(h^{2(k-1)}+(\Delta t)^{2q}\),
\end{equation}
where the constant $C$ does not depend on $\nu^{-1}$. 
\end{Theorem}
\begin{proof}
As in the proof of Theorem~\ref{thpre1}, proving a bound for the pressure starts with the decomposition 
$
p_n-p_h^n =\(p_n-p_h^{s,n}\)  -\(p_h^n-p_h^{s,n}\).
$
Then, instead of \eqref{co_pre1}, one gets, see \cite[Eq.~(42)]{nos_grad_div},
\begin{align}\label{co_pre1b}
\|p_h^n-p_h^{s,n}\|_0 \le& \frac{1}{\beta_0}\left(\|\overline \partial_q\be_h^n\|_{-1}+\nu \|\nabla \be_h^n\|_0+
\sup_{\bv_h\in V_h,\bv_h\neq \boldsymbol 0}\frac{\left|b\(\bs_h^n,\bs_h^n,\bv_h\)-b\(\bu_h^n,\bu_h^n,\bv_h\)\right|}{\|\nabla\bv_h\|_0}\right)\nonumber\\
 &+\frac{1}{\beta_0}\left(\|\btau_1^n+\btau_2^n\|_{-1}+\|\btau_3^n\|_0+\|p^n-\pi_h^n\|_0+\|l_h^n\|_0\right),
\end{align}
where $\btau_1, \btau_2, \btau_3$ come from the error equation \eqref{fully_error_gd_b} and 
$\pi_h^n, l_h^n$ are defined in Section~\ref{sec:notations}.
This relation is squared,  multiplied by $\Delta t$, and then a summation over the time instants is performed.
The bound of the last three extra terms in \eqref{co_pre1b}, compared with \eqref{co_pre1}, is immediate by applying 
\eqref{tau3_nor0b}, \eqref{eq:pressurel2}, and \eqref{eq:cotanewpropre}.

The rest of the proof can be obtained following the steps of the proof of Theorem~\ref{thpre1}. We observe that performing the first step of bounding the temporal derivative of the velocity difference, 
where the G-stability is applied, we arrive, instead of \eqref{eq;derit4}, to 
\begin{eqnarray}\label{eq;derit42}
\lefteqn{\nu \left|\bE^n\right|_{G,1}^2-\nu \left|\bE^{n-1}\right|_{G,1}^2+\mu \left|\bE^n\right|_{G,\tilde 1}^2-\mu \left|\bE^{n-1}\right|_{G,\tilde 1}^2+\Delta t \|\overline \partial_q\be_h^n\|_0^2}\nonumber \\
&\le& \eta_q\Delta t\(\overline \partial_q\be_h^n,\overline \partial_q\be_h^{n-1}\)
 -\Delta t \left[b\(\bs_h^{n},\bs_h^{n},\overline \partial_q\be_h^n\)-b\(\bu_h^{n},\bu_h^{n},\overline \partial_q\be_h^n\)\right]\nonumber\\
&& +\eta_q \Delta t \left[b\(\bs_h^{n-1},\bs_h^{n-1},\overline \partial_q\be_h^n\)- b\(\bu_h^{n-1},\bu_h^{n-1},\overline \partial_q\be_h^n\)\right]
\nonumber\\
&& +\Delta t \(\btau_1^n+\btau_2^n-\eta_q\(\btau_1^{n-1}+\btau_2^{n-1}\),\overline \partial_q\be_h^n\)\nonumber\\
&& +\Delta t \(\btau_3^n-\eta_q \btau_3^{n-1},\nabla \cdot\overline \partial_q\be_h^n\),
\end{eqnarray}
where we have used for the grad-div term the notation
\[
\left|\bV^n\right|_{G,\tilde 1}^2=\sum_{i,j=1}^q g_{i,j}\left(\nabla \cdot v^{n-i+1},\nabla \cdot v^{n-j+1}\right).
\]
Bounding the last term on the right-hand side of \eqref{eq;derit42} with the inverse inequality \eqref{inv}
gives the term $\|\nabla \cdot\overline \partial_q\be_h^n\|_0\le {c_{\mathrm{inv}}}h^{-1}\|\overline \partial_q\be_h^n\|_0$, which leads to the extra
term 
$
C\Delta t \sum_{n=q}^M h^{-2}\|\btau_3^n\|_0^2
$
compared with \eqref{eq:derit6}. This term gives the rate of convergence $k-1$ with respect to 
the mesh width in \eqref{eq:pres_bound_2}.
\end{proof}

\section{Numerical studies}\label{sec:7}


{We performed simulations for the example used in \cite{review_nos} ($\nu\in\{10^{-4}, 10^{-6}, 10^{-8}\}, \mu = 0.01$) and obtained results 
that are of the same form as in Figure~11 of this paper. These results, on the one 
hand, show that higher order BDF methods can be used for problems with small viscosity coefficients. 
But on the other hand, we like to present here simulations that better illustrate the practical
use of BDF-q methods.}
Codes for time integration are usually used in practice with variable step sizes and, in the case of 
BDF methods with variable order. To this end, we used a variable-step size, variable-formula BDF code originally written for numerical tests in~\cite{nos_bdf}. The purpose is to check how the time step restrictions \eqref{restri1}, \eqref{restri2},  and \eqref {cota_t_pre}  affect the step size selection.

The developed code follows standard procedures in the numerical integration of ordinary differential equations.  For the implementation of the methods, we used a strategy as described in~\cite[\S~III.5]{Hairer}.
At each time level $t_n$ (after the second one) the local (time discretization) error is estimated with the quantity
$$
{\rm EST}_n=\frac{\Delta t_n}{t_{n+1}-t_{n-q}} \left\|{\cal U}^{n,q}_h\right\|_0,
\quad \mbox{where}\quad  
{\cal U}^{n,q}_h= \prod_{i=0}^{q-1}\(t_{n+1} -t_{n-i}\) \bu_h[t_{n+1},\ldots,t_{n-q}],
$$
and $\bu_h[t_{n+1},\ldots,t_{n-q}]$ is the standard $(q+1)$-th divided difference based on $t_{n+1},\ldots,t_{n-k}$ and
$\bu_h^{n+1},\ldots, \bu_h^{n-k}$, if the method of order~$q$ is used at~$t_n$.
The estimation ${\rm EST}_n$ is compared with
$$
{\rm TOL}_n ={\rm TOL}_r \left(\max\left\{\left\| \bu_h^{n+1}\right\|_0,\left\| \bu_h^{n}\right\|_0\right\}+0.001\right),
$$
where ${\rm TOL}_r$ is a given tolerance. If ${\rm EST}_n>{\rm TOL}_n$, the computed approximation $\bu_h^{n+1}$ is not considered to be sufficiently accurate and is rejected. It is then recomputed from $t_n$ with a new step length given by
\begin{equation}
\label{eq:change}
\Delta t_{n}^{\rm new} = 0.9\Delta t_n \sqrt[q+1]{\frac{{\rm TOL}_n}{{\rm EST}_n}}.
\end{equation}
On the contrary, if ${\rm EST}_n\le {\rm TOL}_n$, then, $\bu_h^{n+1}$ is accepted and the algorithm proceeds to compute $\bu_h^{n+2}$ with a step length $\Delta t_{n+1}$ given by the right-hand side of~\eqref{eq:change}.

The algorithm starts with~$q=1$ and $\Delta t_0=\Delta t_1=\sqrt{{\rm TOL}_r}/100$. At $t_1$, the first local error estimation~${\rm EST}_1$ is computed. If ${\rm EST}_1> {\rm TOL}_1$ the computation is restarted again from $t_0$ with the step size changed according to~\eqref{eq:change}. From $t_2$ onwards, the estimate corresponding to the BDF-$2$ method can be estimated, and when this estimate is smaller than that corresponding to~$q=1$, the method switches to~$q=2$. From then onwards, estimates for the method of order $q+1$ and~$q-1$ are also computed along with that for the current value of $q$, and the method increases, decreases or maintains the order according to which of the three estimates is the smallest one, but never exceeds a maximum allowed order~$q_{\rm max}$. In the results reported below, we observed that  once the algorithm reached $q_{\rm max}$, that order was maintained, except when $q_{\rm max}=5$ where it fluctuated between different values of~$q$. For this reason, only results for $q_{\rm max}\in\{2,3,4\}$ will be shown below.

We now present results for the well-known benchmark problem defined in~\cite{bench}. The domain is given by
$$
\Omega=(0,2.2)\times(0,0.41)/\left\{ (x,y) \mid (x-0.2)^2 + (y-0.2)^2 \le 0.0025\right\}
$$
and the time interval is~$[0,8]$. The velocity is identical on both vertical sides, prescribed by
$$
\bu(0,y)=\bu(2.2,y)=\frac{6}{0.41^2}\sin\(\frac{\pi t}{8}\)\left(\begin{array}{c} y(0.41-y)\\ 0\end{array}\right).
$$
On the remainder of the boundary $\bu ={\bf 0}$ is set. At~$t=0$, the initial velocity is ${\bu}={\bf 0}$. The viscosity is chosen to be $\nu=10^{-3}$
and the forcing term is $\bff={\bf 0}$.  We used quadratic elements for the velocity and linear ones for the pressure, on the same mesh as in~\cite{nos_bdf},
which has~6624 triangular mesh cells with diameters ranging from $5.53\times 10^{-3}$ to~$3.38\times 10^{-2}$. The number of degrees of freedom is~27168 for the velocity
and~3480 for the pressure. The grad-div parameter was set to~$\mu=0.01$, as suggested by numerical experience in~\cite{proy_gdiv}.

\begin{figure}[h]
\begin{center}
\includegraphics[height=2.6truecm]{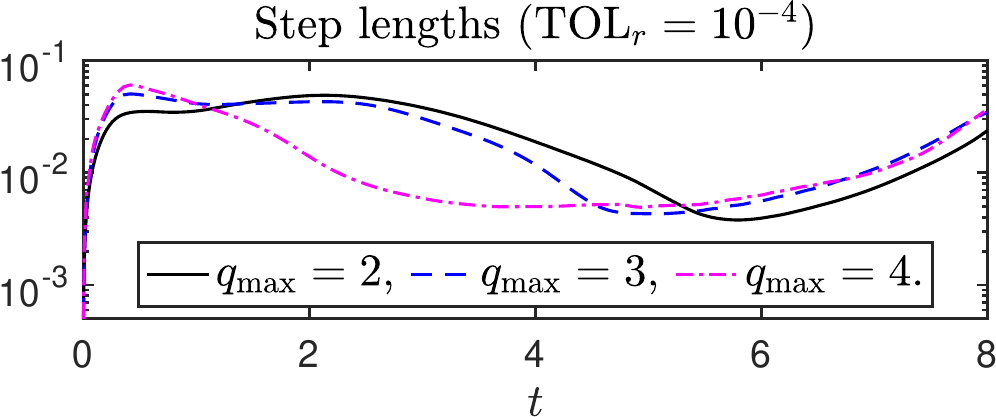} \quad 
\includegraphics[height=2.6truecm]{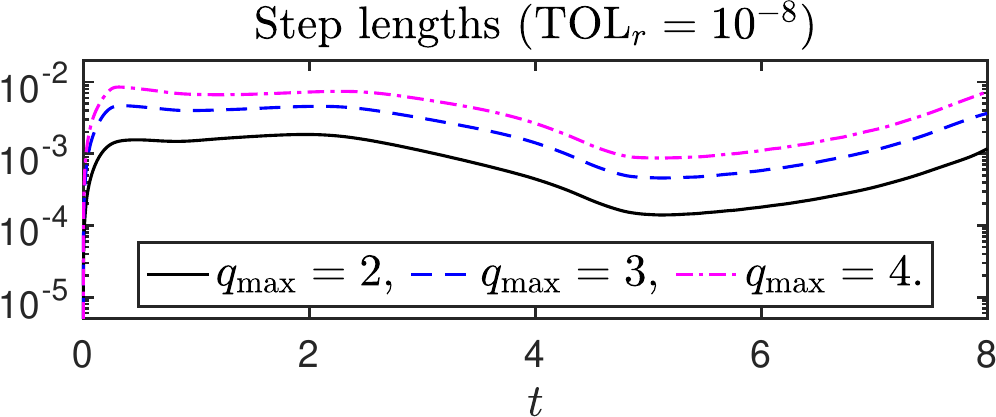}
\caption{\label{fig_hcyl} Step lengths for variable-formula, variable-order code for different values of~$q_{\rm max}$ and tolerances ${\rm TOL}_r=10^{-4}$ (left) and~${\rm TOL}_r=10^{-8}$ (right).}
\end{center}
\end{figure}

In Fig.~\ref{fig_hcyl} we show the step lengths used by the code when $q_{\rm max}$ is set to~$2$, $3$, and~$4$. The left plot corresponds to~${\rm TOL}_r=10^{-4}$, and the right one to~${\rm TOL}_r=10^{-8}$. We see that whereas for~${\rm TOL}_r=10^{-8}$ the step lengths behave as one could expect, that is, the higher the order allowed in the code, the larger the steps are taken, there is an interval in the case of~${\rm TOL}_r=10^{-4}$ where the opposite happens. The explanation is that for~${\rm TOL}_r=10^{-8}$ the step lengths selected by the code are so small that they are well inside the restrictions~\eqref{restri1}, \eqref{restri2},   and \eqref {cota_t_pre}, while with ${\rm TOL}_r=10^{-4}$ the step lengths, being larger, may violate some of these restrictions, resulting in larger errors, which in turn are detected by the estimation~${\rm EST}_n$ that forces the algorithm to reduce the step length accordingly. This phenomenon was also observed in~\cite{nos_bdf} with an IMEX-BDF-$2$  method, which suffers a time step restriction similar to~\eqref{restri2}.

\begin{figure}[h]
\begin{center}
\begin{minipage}{6.5truecm}
\includegraphics[width=6.5truecm]{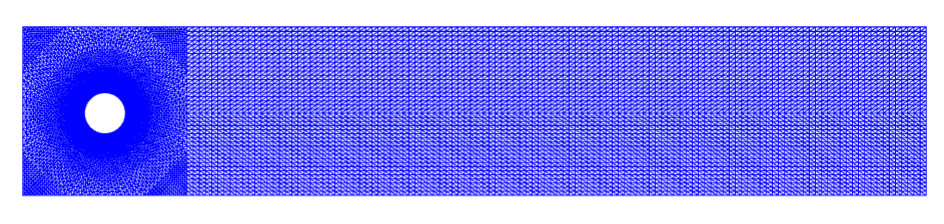}
\end{minipage}
\begin{minipage}{6.5truecm}
\includegraphics[width=6.5truecm]{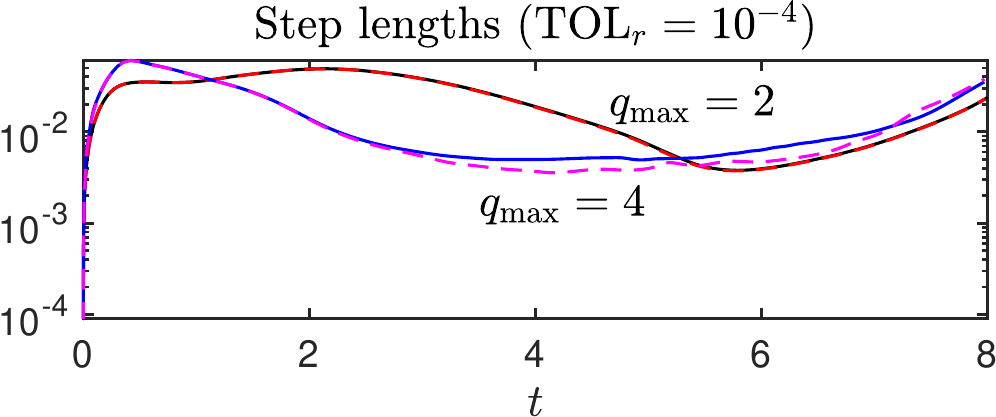}
\end{minipage}

\caption{\label{fig_hcyl2} Left, mesh with 18400 cells. Right, step lengths for~${\rm TOL_r}=10^{-4}$: coarse mesh, $q_{\rm max}=2$ (black continuous line),
$q_{\rm max}=4$ (blue continuos line); fine mesh, $q_{\rm max}=2$ (red discontinuos line), $q_{\rm max}=4$ (magenta discontinuous line).}
\end{center}
\end{figure}

We also checked the influence of the spatial mesh on the step length selection. For this purpose we repeated the computation with the mesh shown in Fig.~\ref{fig_hcyl2} (left plot) which has 18400 cells with diameters ranging from $3.30\times 10^{-3}$ to~$2.09\times 10^{-2}$, so it is finer than the mesh used before. On the right-plot, we show the step lengths on the coarser mesh used before in continuous lines and darker colors and on the finer mesh with discontinuous lines and lighter colors. It can be seen that for BDF-$2$, which does not suffer from time steps restrictions, both lines (continuous and discontinuous) are superimposed, meaning that the step size selection is unaffected by the spatial mesh diameter, whereas for $q_{\max}=4$ slightly smaller steps are taken on the finer mesh in most of the time interval, reflecting the stability restrictions suffered by BDF-$4$. In the case $q_{\rm max}=3$, which is not shown here, the differences between the step lengths on the two meshes were much less pronounced than in the case $q_{\rm max}=4$.

\section{Appendix}

{\bf Proof of Lemma \ref{lema_bosco}} In order to avoid confusion with powers, the time level will be denoted
as a subscript in this proof, in contrast to the remainder of the paper. To prove \eqref{cota_Aesta}, we have to 
study the methods in terms of the variables $D y_{n} = y_{n} - y_{n-1}, \ldots, Dy_{n-q+1} = y_{n-q+1}-y_{n-q}$. 
It is well known that the BDF formula of order~$q$ for $y'=f(t,y)$ can be written as
\begin{equation}\label{bbdfq}
\sum_{j=1}^q \frac{1}{j} D^j y_{n} = \Delta t f_{n},
\end{equation}
where $f_{n} = f(t_{n},y_{n})$.
Since $D^j = D^{j-1} D$ and
$$
D^{j-1} z_{n} = \sum_{k=0}^{j-1}(-1)^k {{j-1}\choose{k}} z_{n-k},
$$
the left-hand side of~\eqref{bbdfq} can be reformulated as
$$
D y_{n}+ \sum_{j=2}^{q} \frac{1}{j}\sum_{k=0}^{j-1} (-1)^k {j-1\choose k} Dy_{n-k}=
D y_{n}+ \sum_{k=0}^{q-1}\gamma_{q,k}Dy_{n-k},
$$
where
\begin{equation}
\label{coefs_bgamma}
\gamma_{q,k} =(-1)^k\( \sum_{j=k+1, {j>1}}^q \frac{1}{j} {j-1\choose k} \), \quad k=0,\ldots, q-1.
\end{equation}
Table~\ref{table_bgammas} shows the values of the coefficients $\gamma_{q,k}$.
Thus, the BDF formula of order~$q$ can be written as
\begin{equation*}
Dy_{n} + \sum_{k=0}^{q-1} \gamma_{q,k} Dy_{n-k} = \Delta t f_n.
\end{equation*}

\begin{table}[t!]
\begin{center}
\caption{\label{table_bgammas} Coefficients $\gamma_{q,k}$ in~\eqref{coefs_bgamma}}
$$
\begin{array}{|c|c|c|c|c|c|c|}
 \hline
q\backslash k & 0 & 1& 2& 3& 4\\  \hline
3 & 5/6 & -7/6 & 1/3 & & \\  \hline
4 & 13/12 & -23/12 & 13/12 & -1/4 & \\  \hline
5 & 77/60 & -163/60 & 137/60& -63/60 & 1/5\\  \hline
\end{array}
$$
\end{center}
\end{table}

We study now the $G$-stability of the polynomials
$
x^{q-1} + p_q(x)$, $q\in\{3,4,5\}$, where
$$
p_q(x)=\gamma_{q,0}x^{q-1} + \gamma_{q,1} x^{q-2} + \cdots + \gamma_{q,q-2} x + \gamma_{q,q-1},
$$
or to be more precise, for $\mu_q(x)= x^{q-1} $,
we compute
$$
\sigma_q=-\min_{\theta\in[-\pi,\pi]}\hbox{Re}\left( \frac{p_q\left(e^{i\theta}\right)}{\mu_q\left(e^{i\theta}\right)}\right) = -\min_{\theta\in[-\pi,\pi]}\left( \gamma_{q,0} + \gamma_{q,1}\cos(\theta) + \cdots+ \gamma_{q,q-1} \cos((q-1)\theta)\right).
$$
If $\sigma_q<1$, then, we may write the polynomial $x^{q-1}+p_q(x)$ as the sum $(1-\sigma_q)x^{q-1}$ and~$\sigma_qx^{q-1}+ p_q(x)$, so that 
\[
\Delta t \(\overline \partial_q y_n,D y_n\)_H = \(1-\sigma_q\)\left\|D y_n\right\|_{{H}}^2
+ \(\sigma_q Dy_{n} + \sum_{k=0}^{q-1} \gamma_{q,k} Dy_{n-k},Dy_{n}\).
\]
Considering the method on the right-hand side of this relation, one obtains for its 
root locus curve that 
$$
\min_{\theta\in[-\pi,\pi]}\hbox{Re}\(\frac{\sigma_q \(e^{i\theta}\)^{q-1} + p_q(e^{i\theta})}{\mu_q(e^{i\theta})}\) =
\min_{\theta\in[-\pi,\pi]}\hbox{Re}\(\sigma_q+\frac{p_q(e^{i\theta})}{\mu_q(e^{i\theta})}\) \ge 0.
$$
Hence the method is $A$-stable and 
since $A$-stability is equivalent to $G$-stability, see \cite{dalquis2}, the above inequality implies that there exist a symmetric and positive definite matrix~$G_q=(g_{q,i,j})_{1\le i,j\le q-1}$ satisfying \eqref{cota_Aesta} for $s_q=1-\sigma_q$.
To conclude, we compute $\sigma_q$ for $q\in\{3,4,5\}$.

For $q=3$, we have 
$$
\hbox{Re}\( \frac{p_3(e^{i\theta})}{\mu_3(e^{i\theta)}}\) = \frac{5}{6} - \frac{7}{6} \cos(\theta) + \frac{2}{6}\(2\cos^2(\theta) -1\)=
 \frac{3}{6} - \frac{7}{6} \cos(\theta) + \frac{4}{6}\cos^2(\theta).
 $$
 The polynomial $3 -7x +4 x^2$ achieves its minimum at $x=\nicefrac78$,
 so that
 \[
- \sigma_3 = \min_{\theta\in[-\pi,\pi]}\hbox{Re}\( \frac{p_3(e^{i\theta})}{\mu_3(e^{i\theta)}}\) = \frac{3}{6} - \frac{7}{6} \cdot\frac{7}{8} + \frac{4}{6}\cdot \frac{49}{64}=-\frac{1}{96}.
\]
 
If $q=4$, one finds that 
\[
 \hbox{Re}\( \frac{p_4(e^{i\theta})}{\mu_4(e^{i\theta)}}\)  = -\frac{7}{6}\cos(\theta)+ \frac{13}{6}\cos^2(\theta)-\cos^3(\theta).
\]
The polynomial $-7x +13x^2 - 6x^3$ has its minimum when $x\in[-1,1]$ at $x=(13-\sqrt{43})/18$, so that
$
 \sigma_4 = \nicefrac{(260+43\sqrt{43})}{2916}.
$
It is easy to check that, rounded up to five significant digits, $\sigma_4=0.1859$, and consequently,
$
\sigma_4 < \nicefrac{3}{16} < \nicefrac{1}{5}.
$

Finally, for $q=5$, one can prove that
\[
 \hbox{Re}\( \frac{p_5(e^{i\theta})}{\mu_5(e^{i\theta)}}\)= -\frac{24}{30} +\frac{13}{30}\cos(\theta) + \frac{89}{30}\cos^2(\theta)-\frac{126}{30}\cos^3(\theta) +\frac{48}{30}\cos^4(\theta).
\]
 In order to find the minimum for $x\in[-1,1]$ of the polynomial $-24+13x + 89x^2 - 126x^3 +48x^4$, we look for the zeros of its derivative,
 $13 +178x -378x^2 +192 x^3$. We find the minimum for $x= -0.0640410$ rounded to 6 significant digits, so that it is in~the interval~$[-1,1]$. With this value of $x$ on can compute
$\sigma_5$, resulting in
$\sigma_5=0.814454$ rounded to six significant digits. We notice that
$
\sigma_5< \nicefrac{9}{11}.
$

\end{document}